\documentclass[12pt,letterpaper,reqno]{amsart}
\usepackage{graphicx}
\usepackage{amsmath}
\usepackage{amssymb}
\usepackage{amsfonts}
\usepackage{amsrefs}
\usepackage{enumerate}
\usepackage[mathscr]{eucal}

\usepackage[usenames]{color}

\newcommand{\R}{\mathbb{R}}
\newcommand{\Z}{\mathbb{Z}}
\newcommand{\M}{\mathbb{M}}
\newcommand{\F}{\mathcal{F}}
\newcommand{\VF}{\mathcal{VF}}

\newcommand{\rst}{\llcorner}
\newcommand{\toF}{\xrightarrow{\F}}
\newcommand{\toVF}{\xrightarrow{\VF}}

\newtheorem{thm}{Theorem}
\newtheorem{lemma}[thm]{Lemma}
\newtheorem{prop}[thm]{Proposition}
\newtheorem{cor}[thm]{Corollary}

\theoremstyle{definition}
\newtheorem{definition}[thm]{Definition}

\theoremstyle{remark}
\newtheorem{remark}{Remark}
\newtheorem*{ack}{Acknowledgments}
\newtheorem*{outline}{Outline}

\DeclareMathOperator{\Div}{div}

\DeclareMathOperator{\set}{set}

\DeclareMathOperator{\Lip}{Lip}

\def\madm{m_{\mathrm{ADM}}}
\def\miso{m_{\mathrm{iso}}}

\begin{document}

%\title[{Lower semicontinuity of GR mass under intrinsic flat convergence}]{Lower semicontinuity of general relativistic mass under Sormani--Wenger intrinsic flat convergence}
\title[{ADM mass under intrinsic flat convergence}]{Lower semicontinuity of ADM mass under intrinsic flat convergence}

\author{Jeffrey L. Jauregui}
\address{Dept. of Mathematics,
Union College, 807 Union St.,
Schenectady, NY 12308}
\email{jaureguj@union.edu}
\author{Dan A. Lee}
\address{Graduate Center and Queens College, City University of New York, 365 Fifth Avenue,
New York, NY 10016, USA}
\email{dan.lee@qc.cuny.edu}

\begin{abstract}
A natural question in mathematical general relativity is how the ADM mass behaves as a functional on the space of asymptotically flat 3-manifolds of nonnegative scalar curvature. In previous results, lower semicontinuity has been established by the first-named author for pointed $C^2$ convergence, and more generally by both authors for pointed $C^0$ convergence (all in the Cheeger--Gromov sense). In this paper, we show this behavior persists for the much weaker notion of pointed Sormani--Wenger intrinsic flat ($\F$) volume convergence, under natural hypotheses. We consider smooth manifolds converging to asymptotically flat local integral current spaces (a new definition), using Huisken's isoperimetric mass as a replacement for the ADM mass. 
Along the way we prove results of independent interest about convergence of subregions of $\F$-converging sequences of integral current spaces.
%General results are also given on the $\F$-convergence of subregions in a convergent sequence of integral current spaces that may have other applications.
\end{abstract}

\maketitle

\section{Introduction}
The ADM mass functional \cite{ADM}, defined on the space $\mathcal{M}$ of asymptotically flat 3-manifolds of nonnegative scalar curvature, is of fundamental importance in general relativity. Such manifolds represent physically reasonable time-symmetric initial data sets for Einstein's equation, and the ADM mass defines their total mass. In prior work the authors have studied the continuity behavior of the ADM mass under various notions of pointed convergence. In general, it is expected that the ADM mass is lower semicontinuous, but of course such a statement depends on the topology placed on $\mathcal{M}$. For pointed $C^2$ Cheeger--Gromov convergence this was shown by the first-named author in dimensions three in \cite{Jau} and up to dimension seven in \cite{Jau2}. For pointed $C^0$ (i.e., locally uniform) Cheeger--Gromov convergence, this was shown by the authors in dimension three in \cite{JL}. (The precise statements are recalled in section \ref{sec_background1}.) However, for applications to some outstanding problems in general relativity (discussed below), a coarser topology than $C^0$ is required, perhaps that given by C. Sormani and S. Wenger's intrinsic flat distance \cite{SW}. Since the limit spaces in $\overline{\mathcal{M}}$ need not be smooth, we require a generalization of the ADM mass that is well-defined in lower regularity: we use G. Huisken's isoperimetric mass \cites{Huisken:2006, Huisken:Morse}.

We briefly describe two major open problems to which the lower semicontinuity of mass would be applicable. First, we recall the rigidity statement of the positive mass theorem \cites{SY,W}, which says that a (smooth) asymptotically flat 3-manifold with nonnegative scalar curvature and zero ADM mass is isometric to Euclidean space. A natural, well-known conjecture is: if a sequence $M_j$ of asymptotically flat 3-manifolds with nonnegative scalar curvature has ADM mass $m_j$ converging to zero as $j \to \infty$, then the $M_j$ must converge to Euclidean space (in an appropriate topology). Unfortunately, the pointed $C^k$ Cheeger--Gromov (even with $k=0$) and pointed Gromov--Hausdorff topology are insufficient (see \cites{LS,S16}, for instance). The second-named author and Sormani conjectured in \cite{LS} that this conjecture holds for the topology given by the pointed intrinsic flat distance. They proved this for the rotationally symmetric case, and subsequent work by L.-H. Huang, the second-named author, and Sormani \cite{HLS}, A. Sakovich and Sormani \cite{SakSor}, Sormani and I. Stavrov Allen \cite{SorSta}, and B. Allen \cite{All} has confirmed this conjecture (for pointed intrinsic flat convergence) in a number of other special cases. One possible program to attack the conjecture in general is to attempt to extract a pointed intrinsic flat limit $M_\infty \in \overline{\mathcal{M}}$ of the $M_j$ (upon taking a subsequence). If the total mass is lower semicontinuous, then the limit space $M_\infty$ would have nonpositive total mass (if appropriately defined). Finally, if a weak version of the positive mass theorem (together with a rigidity statement) can be obtained on $\overline{\mathcal{M}}$, it would follow that the original sequence $M_j$ converges to Euclidean space, as conjectured.

Second, we recall R. Bartnik's mass-minimization conjecture. Let $\Omega$ be a compact Riemannian 3-manifold of nonnegative scalar curvature with boundary diffeomorphic to $S^2$. Bartnik's quasi-local mass \cite{Ba1} was originally defined by taking the infimum of the ADM mass among asymptotically flat 3-manifolds $M$ of nonnegative scalar curvature in which $\Omega$ embeds isometrically, provided $M$ contains no horizons (compact minimal surfaces). (Many other variants of Bartnik's definition have appeared since; we refer the reader to \cite{Jau3} for a recent discussion.) Bartnik's mass-minimization conjecture \cite{Ba1} is that the infimum is achieved, at least under some hypotheses on $\Omega$ (see \cite{AJ}). A direct approach to this conjecture would entail taking a sequence $M_j$ of such extensions of $\Omega$ and hoping to extract a convergent subsequence in some topology, say with limit $M_\infty$, which may be a priori non-smooth. (In light of the results discussed in the previous paragraph, a natural candidate may be the pointed Sormani--Wenger intrinsic flat topology.)  In any case, to show that $M_\infty$ is indeed a Bartnik mass minimizer, one would need to know that the ADM mass (or an appropriate generalization) is lower semicontinuous when passing to a limit.

\medskip

We also mention here the observation in \cite{Jau} that the lower semicontinuity of ADM mass for pointed convergence, even in the $C^2$ case, recovers the positive mass theorem via a simple blowup example. Thus, there is a direct connection between the lower semicontinuity and the positivity of mass. We remark that the positive mass theorem enters into the proof of the main theorem below via the use of Theorem \ref{thm17}, which uses Huisken and T. Ilmanen's results on weak inverse mean curvature flow \cite{HI} (which imply the positive mass theorem).

Our main theorem is:
\begin{thm}
\label{thm1}
Let $(M_j, g_j, p_j)$  be a sequence of smooth, oriented asymptotically flat Riemannian 3-manifolds without boundary, with nonnegative scalar curvature, containing no 
compact minimal surfaces. Assume there exists a uniform positive lower bound for the isoperimetric constants of $(M_j, g_j)$. If this sequence converges in the pointed intrinsic flat volume sense to a complete pointed asymptotically flat local integral current space  $N=(X,d,g,T,q)$ of dimension 3 as $j \to \infty$, then
$$\liminf_{j \to \infty} \madm(M_j,g_j) \geq \miso(N),$$
where $\miso(N)$ is the isoperimetric mass of $N$.
\end{thm}
The relevant definitions are given in sections \ref{sec_background1} and \ref{sec_background2}.

We regard Theorem \ref{thm1} as additional new evidence that $\F$-convergence interacts well with nonnegative scalar curvature and general relativistic mass.

\medskip

Since much of the paper is technical, we describe here the outline of the proof of Theorem \ref{thm1} to explain why the technical details are needed. Suppose we have a sequence $(M_j, g_j, p_j)$ as above, converging to a limit space $N$ (to be interpreted vaguely for now). By its definition the quantity  $\miso(N)$ will be approximated by $\miso(E)$, where $E$ is a large compact set in the limit space, and the ``quasi-local'' isoperimetric mass of $E$ is defined by
$$\miso(E) = \frac{2}{A} \left( V - \frac{1}{6\sqrt{\pi}} A^{3/2}\right),$$
where $A$ is the boundary area (perimeter) of $E$ and $V$ is the volume of $E$. (To establish some intuition, note that if $E \subset \R^3$, then $\miso(E) \leq 0$ by the isoperimetric inequality, equaling zero if $E$ is a ball. Also if $E$ is a ball in the Schwarzschild manifold of mass $m>0$, then $\miso(E)$ limits to $m$ as the radius of the ball limits to infinity.)  The next step is to construct compact sets $E_j \subset M_j$ that correspond to $E$ in a meaningful way. This is immediate for pointed Cheeger--Gromov convergence, which comes equipped with embeddings from the limit space into $M_j$; for intrinsic flat convergence, however, there are no such maps, through there are maps from $X$ and $M_j$ into a common metric space by a theorem of Sormani and Wenger. Using these maps and properties of Lipschitz functions, we construct $E_j$. To proceed, we want it to be the case that $\miso(E_j)$ approximates $\miso(E)$ for large $j$, which is guaranteed if the boundary area $A_j$ and volume $V_j$ of $E_j$ are close to $A$ and $V$. Assume this is true for a moment (although see the next paragraph). It would be convenient if the quasi-local mass $\miso(E_j)$ gave a lower bound on the ADM mass of $M_j$; \cite[Theorem 17]{JL} (see Theorem \ref{thm17} herein) nearly does so, instead giving:
$$\miso(E_j) \leq m_{ADM}(M_j, g_j) + \frac{C}{\sqrt{A_j}}.$$
Here, $C$ is a constant that we can control independently of $j$, and we can pre-arrange that $A$ and hence $A_j$ are large enough to make $\frac{C}{\sqrt{A_j}}$ arbitrarily small. Then, for $j$ large,
$$\miso(N) \approx \miso(E) \approx \miso(E_j) \lessapprox m_{ADM}(M_j, g_j),$$
which would imply the result.

The weakest link in the above sketch is that $A_j \approx A$.  Since we assume intrinsic flat \emph{volume} convergence, we will be able to show that $V_j \approx V$. Unfortunately, under this mode of convergence, the boundary areas will only be lower semicontinuous, i.e. can jump strictly down in a limit. Such behavior would cause $\miso(E_j)$ to jump \emph{up} in a limit, which unfortunately would be incompatible with the argument above. The main technical work of this paper is to construct the $E_j$ carefully so as to obtain convergence of the boundary areas of $E_j$ to that of $E$. This involves a double perturbation argument; $E$ and $E_j$ will be viewed as sub-level sets of Lipschitz functions, and it will be necessary to vary both the level set value and also a cut-off radius to prevent $E_j$ from extending too far out into $M_j$. We also require a precise definition of the possible limit space; Definition \ref{def_AF_LICS} allows for a rough (metric space) limit, provided it is a $C^0$ Riemannian manifold outside a compact set.

\begin{outline}
Section \ref{sec_background1} includes the relevant background for the paper on $C^0$ asymptotically flat manifolds, Huisken's isoperimetric mass, and precise statements of prior results. Section \ref{sec_background2} recalls the important details of Ambrosio--Kirchheim  integral currents, integral current spaces, and the Sormani--Wenger intrinsic flat distance. Section \ref{sec_local_ICS} includes a new definition, that of an asymptotically flat local integral current space. Some concepts defined on integral current spaces (e.g., current mass) must be reconciled with their Riemannian counterparts (e.g., volume) --- this is carried out in section \ref{sec_measures}. The main technical work of the paper is in section \ref{sec_regions}, which includes the construction of the regions $E_j$ mentioned above and proofs of the convergence of their volumes and perimeters. Theorem \ref{thm1} is finally proved in section \ref{sec_main_proof}; a discussion of the hypotheses in Theorem \ref{thm1} follows the proof. The appendix includes a discussion of how perimeter is defined in both the smooth and the $C^0$ setting and proves some needed results.
\end{outline}

\begin{ack}
The authors would like to thank Christina Sormani for valuable discussions and support. 
JJ acknowledges support from Union College's Faculty Research Fund.
\end{ack}

\section{Background I: $C^0$ asymptotically flat manifolds; Huisken's isoperimetric mass; prior results}
\label{sec_background1}

\subsection{$C^0$ Riemannian manifolds and asymptotic flatness}

\begin{definition}
For an integer $k \geq 0$, a \emph{$C^k$ Riemannian manifold} is a smooth manifold $M$ (possibly with boundary), equipped with a $C^k$ Riemannian metric $g$.
\end{definition}

Any connected $C^k$ Riemannian manifold $(M,g)$ naturally has a distance function $d_g$ that induces the manifold topology; $d_g(x,y)$ is found by taking the infimum of the lengths of piecewise smooth curves joining $x$ to $y$. This is of course well-known for $k \geq 2$; for $k \geq 0$, see the work of A. Burtscher \cite{Bur}, who also shows that absolutely continuous curves may be used instead without altering the definition. We say $(M,g)$ is complete if the metric space $(M, d_g)$ is complete. We let $B_g(p,r)$ denote an open ball in $M$ with respect to $d_g$. 

On a $C^k$ Riemannian manifold $M$, $k \geq 0$, of dimension $n$, there are well-defined notions of Lebesgue measure and Hausdorff measure. For the former, the volume of a measurable set can be computed locally by integrating $\sqrt{\det(g_{ij})}dx^1 \ldots dx^n$ in a coordinate chart. Moreover, there is a notion of the perimeter of a set; we recall this in the appendix. For now, we simply recall that if $E \subset M$ has smooth boundary, then its perimeter equals the Hausdorff $(n-1)$-measure of its boundary.  We will denote the volume and perimeter of a set $E \subset M$ with respect to $g$ by, respectively, $|E|_g$ and $|\partial^*E|_g$, or $|E|$ and $|\partial^* E|$ if the metric is understood.

Next we recall a natural notion of convergence for sequences of $C^k$ Riemannian manifolds.
\begin{definition}
\label{def:CG}
Fix an integer $k \geq 0$. A sequence of complete, connected pointed $C^k$ Riemannian $n$-manifolds $(M_j,g_j,p_j)$  converges  \emph{in the $C^k$ pointed Cheeger--Gromov sense} to a complete pointed $C^k$ Riemannian $m$-manifold $(N,h,q)$ if, given any $R> 0$, there exists an open set $U$ in $N$ containing the ball $B_h(q,R)$ and smooth embeddings $\Phi_j : U \longrightarrow M_j$ (for all $j$ sufficiently large) whose image contains the ball $B_{g_j}(p_j,R)$ in $M_j$, such that the sequence of tensors $\Phi_j^* g_j$ converges in $C^k$ to $h$ on $U$, as $j \to \infty$.
\end{definition}

We now define asymptotic flatness in both the continuous  and smooth cases.

\begin{definition} A \emph{$C^0$ asymptotically flat (AF) end} is a $C^0$ Riemannian $m$-manifold $(M,g)$ with boundary, where $n \geq 3$, for which there exists a diffeomorphism $\Phi$ from $M$ to $\R^m$ minus an open ball, such that in the coordinates $(x^i)$ determined by $\Phi$,
\begin{equation}
\label{eqn_decay}
|g_{ij} - \delta_{ij}| = O(|x|^{-p}),
\end{equation}
for some constant $p > 0$, where $|x| = \sqrt{(x^1)^2 + \ldots + (x^n)^2}$.
 A \emph{$C^0$ asymptotically flat (AF) manifold} is a connected $C^0$ Riemannian manifold for which there exists a compact set $K$, for which the closure of $M \setminus K$ is a $C^0$ AF end.
\end{definition}

Smooth AF ends and smooth AF manifolds are defined as above, with the additional requirements that $g$ be smooth; that \eqref{eqn_decay} holds with $p={n-2}$; further decay of the partial derivatives
\begin{align*}
\partial_k g_{ij}&=O(|x|^{1-n})\\
\partial_k \partial_l g_{ij}&=O(|x|^{-n})
\end{align*}
holds;  and the scalar curvature of $g$ is integrable.

\subsection{ADM mass and Huisken's isoperimetric mass}

For smooth AF ends of dimension $m$, the ADM mass \cite{ADM} is well-defined  \cites{Bar,Chr} by the formula
$$m_{ADM}(M,g) = \frac{1}{2(n-1)\omega_{n-1}} \lim_{r \to \infty} \int_{S_r} \sum_{i,j=1}^n\left( \partial_ig_{ij} - \partial_j g_{ii}\right)\frac{x^j}{r} dA ,$$
where $dA$ is the induced volume form on the coordinate sphere $S_r=\{|x|=r\}$ with respect to the Euclidean metric $\delta_{ij}$, all in a coordinate chart for which the appropriate decay holds. This real number represents the total mass ``seen'' from the AF end.  For $C^0$ AF ends, the ADM mass need not be defined at all. As in \cite{JL}, we use G. Huisken's isoperimetric mass concept \cites{Huisken:Morse, Huisken:2006} to endow $C^0$ AF ends with a notion of total mass. \emph{We now restrict to dimension three.}

\begin{definition}
\label{def:isop}
Let $(M,g)$ be a $C^0$ Riemannian 3-manifold and $\Omega \subset M$ a bounded open subset of finite perimeter (see the appendix). The \emph{isoperimetric ratio} of $\Omega$ is
\[I(\Omega) = I(\Omega,g ) = \frac{|\partial^* \Omega|_g^{3/2}}{|\Omega|_g}.\]
The \emph{isoperimetric constant} of $(M,g)$ is the infimum of the isoperimetric ratios of all such sets $\Omega$, denoted $c(M,g)$.
\end{definition}

In \cite[Lemma 8]{JL} it is shown that $c(M,g)>0$ for a $C^0$ AF manifold. Note that by the classical isoperimetric inequality, $I(\Omega) \geq 6\sqrt{\pi}$ for all bounded open subsets of finite perimeter in Euclidean 3-space, with equality precisely on balls.

\begin{definition}[\cites{Huisken:Morse, Huisken:2006}]
\label{def:isop_mass}
Huisken's \emph{quasilocal isoperimetric mass} of $\Omega$ is 
\[ \miso(\Omega)=\miso(\Omega,g):=   \frac{2}{|\partial^* \Omega|_g}\left(|\Omega|_g - \frac{1}{6\sqrt{\pi}}|\partial^* \Omega|_g^{3/2} \right).\] 
If $(M,g)$ is a $C^0$ AF end (or more generally a $C^0$ AF manifold), then Huisken's  \emph{isoperimetric mass} of $(M,g)$ is defined by
\[\miso(M,g) = \sup_{\{\Omega_i\}_{i=1}^\infty} \left(\limsup_{i \to \infty} \miso(\Omega_i,g)\right),\]
where the supremum is taken over all exhaustions $\{\Omega_i\}_{i=1}^\infty$ of $M$ by bounded open subsets of finite perimeter containing $\partial M$.
\end{definition}
Note that $\miso(M,g)$ takes values in $[-\infty,\infty]$ and is independent of any choice of coordinates.

That this is a good definition follows from the fact that $\miso(M,g)$ equals $\madm(M,g)$ if $g$ is smooth with nonnegative scalar curvature. Huisken announced this equality when he first introduced his isoperimetric mass concept (see, e.g., \cite{Huisken:Morse}). P.\ Miao observed 
that the volume estimates of X.-Q.\ Fan, Y.\ Shi, and L.-F.\ Tam \cite{Fan-Shi-Tam:2009} imply $\miso(M,g)\ge \madm(M,g)$. More recently, the reverse inequality follows from \cite[Theorem 17]{JL} (which is essentially restated as Theorem \ref{thm17} below) or the work of O. Chodosh, M. Eichmair, Y. Shi, and H. Yu \cite{Chodosh-Eichmair-Shi-Yu:2016}.

For later reference, we recall that a bounded open set of finite perimeter $\Omega \subset M$ (where $(M,g)$ is a $C^0$ AF manifold) is \emph{outward-minimizing} if $|\partial^*\Omega| \leq |\partial^*\Omega'|$ for all bounded open sets of finite perimeter $\Omega' \supset \Omega$ in $M$.

\subsection{Prior results on the lower semicontinuity of mass}
In prior work the authors proved:

\begin{thm}[Lower semicontinuity of mass under $C^0$ convergence \cite{JL}]
\label{thm_JL}
Let $(M_j, g_j, p_j)$ be a sequence of pointed smooth asymptotically flat 3-manifolds whose boundaries are empty or minimal, such that each 
$(M_j ,g_j)$ has nonnegative scalar curvature and contains no compact minimal surfaces in its interior. If $(M_j, g_j, p_j)$ converges in the pointed $C^0$ Cheeger--Gromov sense to a pointed $C^0$ asymptotically flat 3-manifold $(N, h, q)$, then 
\[ \liminf_{j\to\infty} \madm(M_j, g_j)\geq \miso(N, h),\]
where $\miso$ is Huisken's isoperimetric mass.
\end{thm}

Theorem \ref{thm_JL} was preceded by a result of the first-named author in \cite{Jau} that such lower semicontinuity of the ADM mass holds for pointed $C^2$ Cheeger--Gromov convergence, where the limit space $(N,h)$ is a smooth AF manifold, and $\miso(N,h)$ is replaced with $\madm(N,h)$. (Note, however, that the case of a nonempty (minimal) boundary was not addressed in \cite{Jau}.) It was also shown that the hypotheses of nonnegative scalar curvature and no minimal surfaces are necessary. Later, the higher dimensional case, up to $n=7$, was proved in \cite{Jau2}, provided the limit is asymptotically Schwarzschild.

In \cite[Theorem 12]{Jau}, it was also shown that the ADM mass is lower semicontinuous on the space of smooth, rotationally symmetric $n$-manifolds with respect to a type of pointed intrinsic flat convergence that assumed uniformly bounded ``depth.'' (Volume convergence was not assumed, but in rotational symmetry with a depth bound, it follows by \cite[Theorem 8.1]{LefSor}.) Therein it was conjectured that lower semicontinuity of the ADM mass with respect to pointed intrinsic flat converge ought to hold; Theorem \ref{thm1} may be regarded as significant progress on establishing this.

We do not make direct use of Theorem \ref{thm_JL}; instead we will use the following result, proven in \cite{JL}, which shows the quasilocal isoperimetric mass gives a lower bound for the ADM mass, modulo a controlled error term.

\begin{thm} \label{thm17}  \cite{JL}
Given constants $\mu_0 > 0$, $I_0 > 0$, $c_0>0$, there exists a constant $C=C(\mu_0, I_0,c_0)>0$ with the following property. Let $(M,g)$ be a smooth asymptotically flat 3-manifold without boundary, with nonnegative scalar curvature and containing no compact minimal surfaces, with $\madm(M,g) \leq \mu_0$. Let $\Omega$ be an outward-minimizing bounded open set in $M$ with $C^{1,1}$ boundary $\partial \Omega$. Assume that $|\partial \Omega| \geq 36\pi \mu_0^2$, the isoperimetric ratio of $\Omega$ is at most $I_0$, and the isoperimetric constant $c(\Omega,g)$ is at least $c_0$. Then \[ \miso(\Omega) \leq \madm(M,g) + \frac{C}{\sqrt{|\partial \Omega|}}. \]
\end{thm}

\section{Background II: Ambrosio--Kirchheim currents and the Sormani--Wenger intrinsic flat distance}
\label{sec_background2}

We will heavily use Ambrosio--Kirchheim's notion of currents on metric spaces \cite{AK}, as well as Sormani--Wenger's notions of integral current spaces and of intrinsic flat convergence \cite{SW}. We recall a number of relevant definitions and results now, although we also refer readers to the excellent survey of the subject in \cite{S16}. This section is independent of the previous until Definition \ref{def_AF_LICS}.

Throughout this section, $(X,d)$ will denote a  metric space that is assumed to be complete, until indicated otherwise.

\subsection{Ambrosio--Kirchheim currents on metric spaces}
A current on $(X,d)$ (or simply a current on $X$), as defined by Ambrosio--Kirchheim, is, roughly, a multilinear functional acting on tuples of Lipschitz functions $X \to \R$, obeying the properties of continuity, locality, and finite mass. To make this precise, for an integer $m \geq 0$, let $ \mathcal{D}^m$ denote the vector space of $(m+1)$-tuples $(f, \pi_1, \ldots, \pi_m)$, where $f:X \to \R$ is Lipschitz and bounded, and each $\pi_i:X \to \R$ is Lipschitz for $i=1,\ldots, m$. The motivation is to view this $(m+1)$-tuple as the differential $m$-form $f d\pi_1 \wedge \ldots \wedge d\pi_m$.

\begin{definition} \cite{AK}
\label{def_current}
For an integer $m \geq 0$, an \emph{$m$-dimensional current} (or \emph{$m$-current}) on $(X,d)$ is a multilinear function $T: \mathcal{D}^m  \to \R$, obeying:
\begin{enumerate}[(i)]
\item (\emph{continuity}) $T(f,\pi_1, \ldots, \pi_m) = \displaystyle \lim_{j \to \infty} T(f,\pi_1^j, \ldots, \pi_m^j)$ whenever $\pi_i^j \to \pi_i$ as $j \to \infty$ pointwise with uniformly bounded Lipschitz constants.
\item (\emph{locality}) $T(f, \pi_1 \ldots, \pi_m)=0$ if any of the $\pi_i$ is constant on a neighborhood of $\{f\neq 0\}$.
\item (\emph{finite mass}) There exists a finite Borel measure $\mu$ on $X$ such that
\begin{equation}
\label{eqn_mass_T}
|T(f, \pi_1, \ldots, \pi_m)| \leq  \int_X |f| d\mu
\end{equation}
whenever $\Lip(\pi_i)\leq 1$ for all $i=1,\ldots,m$.
\end{enumerate}
\end{definition}

The \emph{mass measure} of a current $T$ is the minimal Borel measure satisfying \eqref{eqn_mass_T}, denoted $\|T\|$. The \emph{mass} of $T$, denoted $\M(T)$, is $\|T\|(X)$, finite by definition.  The \emph{canonical set} of $T$ is comprised of the points in $X$ at which $\|T\|$ has positive lower density (cf. \cite[Theorem 4.6]{AK}):
\begin{equation}
\label{set_T}
\set(T) = \left\{ p \in X \; \Big| \; \liminf_{r \to 0} \frac{\|T\|(B(p,r))}{r^m} > 0 \right\},
\end{equation}
where $B(p,r)$ denotes an open metric ball in $X$.

The standard operations on classical (de Rham) currents on $\R^n$ (e.g., boundary, restriction, slice, push-forward) naturally carry over to the present setting. We describe these in detail now, continuing to follow \cite{AK}.

\smallskip
\paragraph{\emph{Boundary:}}
If $T$ is an $m$-current, $m \geq 1$, its \emph{boundary} $\partial T$ is the multilinear functional on $\mathcal{D}^{m-1}$ given by
$$\partial T(f, \pi_1, \ldots, \pi_{m-1}) = T(1,f, \pi_1, \ldots, \pi_{m-1}).$$
(If $T$ is a $0$-current, let $\partial T$ be the zero $0$-current.) Note that $\partial T$ need not be an $(m-1)$-current; if it is,  $T$ is called a \emph{normal current} (\cite[Definition 3.4]{AK}). In all cases, $\partial(\partial T)=0$.	

\smallskip
\paragraph{\emph{Restriction:}}
We recall the definition of the restriction of an $m$-current $T$ in a few special cases only. First, if $g:X \to \R$ is a bounded, Lipschitz function, define the multilinear functional on $\mathcal{D}^{m}$ by
$$(T \rst g)(f,\pi_1, \ldots, \pi_m) = T(fg,\pi_1, \ldots, \pi_m).$$
In fact, $T \rst g$ is an $m$-dimensional current. Second, to restrict $T$ to a Borel-measurable set $A \subset X$, we replace $g$ with an indicator function $\chi_A$. Although $\chi_A$ is not Lipschitz, it is explained in (2.3) of \cite{AK} that a current $T$ can be uniquely extended so that the first argument may be a bounded Borel function. Thus,
$$(T \rst A)(f,\pi_1, \ldots, \pi_m) = T(f\chi_A,\pi_1, \ldots, \pi_m)$$
is a well-defined $m$-dimensional current.

Third, define the restriction of $T$ (for $m\geq 1$) by ``$dg$'' as
$$(T \rst dg)(f,\pi_1, \ldots, \pi_{m-1}) = T(f,g,\pi_1, \ldots, \pi_{m-1}),$$
an $(m-1)$-current. It is straightforward to check that $\|T \rst dg\| \leq \Lip(g)\|T\|$.

\smallskip

\paragraph{\emph{Slicing:}}
Ambrosio--Kirchheim also define the \emph{slice} of a normal $m$-current $T$, $m \geq 1$, by a Lipschitz function $u:X \to \R$ at value $s \in \R$ as the $(m-1)$-dimensional multilinear functional
$$\langle T, u, s\rangle =  -\partial \left(T \llcorner u^{-1}(s,\infty)\right) + (\partial T) \llcorner u^{-1}(s,\infty),$$
which need not be a current in general. It is straightforward to check that the slice can also be written:
$$\langle T, u, s\rangle = \partial \left(T \llcorner u^{-1}(-\infty,s]\right) - (\partial T) \llcorner u^{-1}(-\infty,s].$$
It is also convenient to note:
\begin{align*}
\langle T, -u, -s\rangle &= -\partial \left(T \llcorner u^{-1}[s,\infty)\right) + (\partial T) \llcorner u^{-1}[s,\infty)\\
&= \partial \left(T \llcorner u^{-1}(-\infty,s)\right) - (\partial T) \llcorner u^{-1}(-\infty,s).
\end{align*}

\begin{thm}[Ambrosio--Kirchheim slicing theorem, Theorem 5.6 of \cite{AK} ]
\label{thm_slicing}
If $T$ is an $m$-dimensional normal current on $(X,d)$, $m \geq 1$ and $u:X \to \R$ is Lipschitz, then $\langle T, u, s \rangle$ is a normal $(m-1)$-current for almost every $s \in \R$, and
\begin{align}
\int_{-\infty}^\infty \M \langle T, u, s\rangle ds &= \M(T \rst du) \label{eqn_slicing}\\
&\leq  \Lip(u) \M(T). \nonumber
\end{align}
\end{thm}

\medskip
\paragraph{\emph{Push-forward:}}
Let $(X, d_X)$ and $(Y, d_Y)$ be complete metric spaces, and suppose $\varphi: X \to Y$ is Lipschitz. Given an $m$-current $T$ on $X$, we have an $m$-current $\varphi_{\#}T$ on $Y$ defined by
$$(\varphi_{\#}T)(f, \pi_1, \ldots, \pi_m)  = T(f \circ \varphi, \pi_1 \circ \varphi, \ldots, \pi_m \circ \varphi).$$
From (2.4) of \cite{AK},
$$\|\varphi_{\#} T\| \leq (\Lip(\varphi))^m \|T\|,$$
with equality if $\varphi$ is an isometry.

Next, we recall that push-forwards commute with boundary, restriction, and slicing: it is straightforward to check
\begin{equation}
\label{eqn_bdry_commute}
\partial (\varphi_{\#} T) = \varphi_{\#}(\partial T).
\end{equation}
In particular, if $T$ is a normal current on $X$, then $\varphi_{\#} T$ is a normal current on $Y$. 
If $g:Y \to \R$ is a Lipschitz function,
then \begin{equation}
\label{eqn_restrict_pushforward}
(\varphi_{\#} T) \llcorner dg  = \varphi_{\#}(T \llcorner d(g \circ \varphi)),
\end{equation}
as in \cite[equation (21)]{S14}.  Also, for a Borel set $E \subset \R$,
\begin{equation}
\label{eqn_restrict_pushforward2}
(\varphi_{\#} T) \llcorner g^{-1}(E) = \varphi_{\#} \left(T \llcorner (g \circ \varphi)^{-1} (E)\right),
\end{equation}
as in \cite[equation (22)]{S14}. Finally, although not needed, we remark that it is a straightforward exercise to show that push-forwards commute with slicing, i.e. for all $s \in \mathbb{R}$, 
\begin{equation}
\label{eqn_slice_pushforward}
\varphi_{\#} \langle T, g \circ \varphi, s \rangle = \langle \varphi_{\#} T, g, s\rangle,
\end{equation}
if $T$ is an $m$-current, $m\geq 1$. 

%To prove this, we use \eqref{eqn_restrict_pushforward2} and \eqref{eqn_bdry_commute}:
%\begin{align*}
%\varphi_{\#} \langle T, g \circ \varphi, s \rangle &= \varphi_{\#}\left(\partial \left(T \llcorner (g \circ \varphi)^{-1}(-\infty,s]\right)\right) - \varphi_{\#}\left((\partial T) \llcorner (g \circ \varphi)^{-1}(-\infty,s]\right)\\
%&=\partial\left(\varphi_{\#}  \left(  T \llcorner \varphi^{-1} (g^{-1}(-\infty,s])\right)\right) - \varphi_{\#}\left((\partial T) \llcorner \varphi^{-1}(g^{-1}(-\infty,s])\right)\\
%&=\partial\left(  \left( \varphi_{\#} T\right) \llcorner g^{-1}(-\infty,s]\right) - \left( \varphi_{\#}\partial T\right) \llcorner g^{-1}(-\infty,s]\\
%&= \langle \varphi_{\#} T, g, s \rangle.
%\end{align*}

\smallskip
\paragraph{\emph{Leibniz rule for restriction and boundary:}}
We require one more set of facts regarding the interaction between the restriction and boundary operations:
\begin{lemma}
\label{lemma_leibniz}
Suppose $T$ is a normal $m$-current on $X$, with $m\geq 1$.
\begin{enumerate}
\item[(a)] If $g$ and $h$ are Lipschitz functions on $X$, then
$$\partial((T \rst g) \rst h) = \partial (T \rst g) \rst h + \partial (T \rst h) \rst g - ((\partial T)\rst g)\rst h.$$
\item[(b)] If $A$ and $B$ are Borel subsets of $X$, then
$$\partial((T \rst A) \rst B) = \partial (T \rst A) \rst B + \partial (T \rst B) \rst A - ((\partial T)\rst A)\rst B,$$
provided each of the terms is a current on $X$.
\end{enumerate}
\end{lemma}
\begin{proof}
For (a), consider a $m$-tuple $(f,\pi_1, \ldots, \pi_{m-1})$. We compute the terms one-by-one:
\begin{align*}
\partial((T \rst g) \rst h)(f,\pi_1, \ldots, \pi_{m-1}) &= ((T \rst g) \rst h)(1, f,\pi_1, \ldots, \pi_{m-1}) = T(gh, f,\pi_1, \ldots, \pi_{m-1}).
\end{align*}
Next,
\begin{align*}
(\partial (T \rst g) \rst h)(f,\pi_1, \ldots, \pi_{m-1})  &= \partial (T \llcorner g) (fh, \pi_1, \ldots, \pi_{m-1}) = T(g,fh,\pi_1, \ldots \pi_{m-1}).
\end{align*}
Using a Leibniz rule (see \cite[Theorem 3.5(i)]{AK}), we have
\begin{align*}
(\partial (T \rst g) \rst h)(f,\pi_1, \ldots, \pi_{m-1}) = T(gh,f,\pi_1, \ldots, \pi_{m-1}) + T(fg,h, \pi_1, \ldots, \pi_{m-1}).
\end{align*}
Switching $g$ and $h$, we have
\begin{align*}
(\partial (T \rst h) \rst g)(f,\pi_1, \ldots, \pi_{m-1})  &= T(gh,f,\pi_1, \ldots, \pi_{m-1}) + T(fh,g, \pi_1, \ldots, \pi_{m-1}).
\end{align*}
Again from the Leibniz rule,
\begin{align*}
(((\partial T)\rst g)\rst h)(f,\pi_1, \ldots, \pi_{m-1})   &= T(1,fgh,\pi_1,\ldots \pi_{m-1})\\
&=T(fg,h,\pi_1, \ldots, \pi_{m-1})  + T(fh, g, \pi_1, \ldots, \pi_{m-1})\\
&\qquad\qquad  + T(gh, f, \pi_1, \ldots, \pi_{m-1}).
\end{align*}
From these calculations, (a) follows.

Part (b) follows from (a) upon replacing $g$ and $h$ with the characteristic functions $\chi_A$ and $\chi_B$. The assumption that the terms are currents (and so have finite mass) assures that expressions such as
$$(\partial (T \rst A) \rst B)(f,\pi_1, \ldots, \pi_{n-1})  = T(\chi_A,f\chi_B,\pi_1, \ldots \pi_{m-1})$$
are well-defined, cf. the discussion around (2.3) in \cite{AK}.
%Part (b) will follow from (a) by realizing the indicator functions $\chi_A$ and $\chi_B$ as $L^1(\|T\|)$ limits of uniformly bounded, Lipschitz function sequences $\{g_i\}$ and $\{h_i\}$, respectively. Since $T$ has finite mass, $T$ is continuous in $L^1(\|T\|)$ in the first slot (Theorem 3.5(ii) of \cite{AK}). Since $\partial T$ has finite mass, $T$ is continuous in $L^1(\|T\|)$ in the second slot as well. Thus, the argument from (a) completes the proof.
\end{proof}

\subsection{Integer rectifiable and integral currents}
In order to streamline the discussion of Ambrosio--Kirchheim integer rectifiable currents and integral currents, we follow the approach of \cites{SW,S16} for example and use some of the theorems in \cite{AK} in place of definitions.

Using \cite[Theorem 4.5]{AK}, we say an $m$-current $T$ on $X$, $m \geq 1$ is \emph{integer rectifiable} if there is a sequence  $\{K_i\}_{i=1}^\infty$ of compact subsets of $\R^m$, a 
sequence of $L^1$ functions $\theta_i: \R^m \to \Z$, with $\theta_i$ supported in $K_i$, and maps $\varphi_i: K_i \to X$ that are bi-Lipschitz onto their images, such that
$$T = \sum_{i=1}^\infty T_i \qquad \text{ and }\qquad \sum_{i=1}^\infty \M(T_i) = \M(T),$$
where
$$T_i(f, \pi_1, \ldots, \pi_m) = \int_{K_i} \theta_i (f \circ \varphi_i) d(\pi_1 \circ \varphi_i) \wedge \ldots \wedge d(\pi_m \circ \varphi_i).$$
Note that the integrand in the definition of $T_i$ is a measurable differential form that is defined almost everywhere and integrable, so that the integral is well-defined.

An integer rectifiable $m$-current $T$, $m \geq 1$, is called an \emph{integral current} if $\partial T$ is an $(m-1)$-current (i.e., has finite mass).  By \cite[Theorem 8.6]{AK}, this is equivalent to $\partial T$ itself being integer rectifiable. An integer rectifiable $0$-current is automatically regarded as an integral current.

Integral currents behave well under slicing: if $T$ is an integral current and $u:X \to \R$ is a Lipschitz function, then $\langle T, u, s\rangle$ is an integral current for almost all $s \in \R$ (cf. \cite[Theorem 5.7]{AK}).

A particularly important example is as follows: a compact, oriented, connected $C^0$ Riemannian $m$-manifold $(M,g)$, possibly with boundary, canonically produces an integral $m$-current $T$ on $(M, d_g)$ via integration:
\begin{equation}
\label{eqn_canonical}
T(f,\pi_1, \ldots, \pi_m) = \int_M f d\pi_1 \wedge \ldots \wedge d\pi_m.
\end{equation}
Note that $T$ is independent of the metric, as all $C^0$ Riemannian metrics on $M$ are uniformly equivalent and thus determine identical classes of Lipschitz functions. Of course the mass measure $\|T\|$ depends on $g$.

\subsection{Sormani--Wenger integral current spaces}
We are now in position to recall the definition of an integral current space, due to Sormani and Wenger (cf. \cite[Definitions 2.35 and 2.46]{SW}). At this point, we drop the assumption that $(X,d)$ is complete.

\begin{definition}[\cite{SW}]
For an integer $m \geq 0$, an \emph{integral current space} of dimension $m$ is a triple $N=(X,d,T)$ in which $(X,d)$ is a  metric space and $T$ is an integral $m$-current on the completion $(\bar X, \bar d)$, such that $\set(T) = X$. 
\end{definition}
We also include the zero space of each dimension $m$, with the empty metric space and $T=0$, as an integral current space.

The mass of $N$, denoted $\M(N)$, is defined to be $\M(T)$. If $A \subset X$ is Borel, and if the restriction $T \rst A$ is an integral current, then
$$N \rst A = (\set(T \rst A), \bar d, T \rst A)$$
is an integral current space. If $m\geq 1$, the \emph{boundary} of $N$ is the integral current space
$$\partial N = (\set(\partial T),  \bar d, \partial T).$$

In the case that $T$ is the canonical integral current \eqref{eqn_canonical} of a compact, connected, oriented $C^0$ Riemannian manifold $(M,g)$, we show in Lemma \ref{lemma_mass_measure}(a) that the mass measure $\|T\|$ agrees with the Riemannian volume measure on Borel sets, so in particular $\set(T) = M$, and $(M,d_g,T)$ is an integral current space, with mass equal to the volume of $M$.

Using the Ambrosio--Kirchheim slicing theorem, Sormani, in \cite[Lemma 2.34]{S14}, proved that in an integral current space $N=(X,d,T)$, for every $p \in X$, $T \rst  B(p,r)$ is an integral current for almost every $r>0$, and that
$$B(p,r) \subseteq \set(T \llcorner B(p,r)) \subseteq \bar B(p,r),$$
where $B(p,r)$ and $\bar B(p,r)$ are open and closed balls in $X$. In particular, $N\rst B(p,r)$ is an integral current space for every $p \in X$ and almost every $r>0$.

\subsection{Local integral current spaces and asymptotic flatness}
\label{sec_local_ICS}
One complication for us is that the Ambrosio--Kirchheim definition of current requires finite mass. Working in the asymptotically flat setting requires a generalization of their definition to a type of current that need only have locally finite mass. We use the approach of U. Lang and Wenger, recalling only the main ideas here and referring to their paper for further details \cite{LW}. We relate statements back to the language of Ambrosio--Kirchheim currents whenever possible.

Given a metric space $(X,d)$, Lang and Wenger first consider multilinear functionals $T$ on $(m+1)$-tuples $(f, \pi_1, \ldots, \pi_m)$, where $f$ is a bounded Lipschitz function $X \to \R$ with bounded support, and each $\pi_i:X \to \R$ is Lipschitz on every bounded subset of $X$. Such a $T$ is called a \emph{metric functional} if it satisfies continuity and locality properties analogous to (i) and (ii) in Definition \ref{def_current} (see \cite[Definition 2.1]{LW}).  There is a natural notion of the \emph{mass} of $T$ on any open subset $V \subseteq X$, denoted $\M_V(T) \in [0, \infty]$. $T$ is called an \emph{$m$-dimensional metric current with locally finite mass} provided:
\begin{enumerate}
\item[(a)] $\M_V(T)$ is finite whenever $V \subseteq X$ is bounded and open, and
\item[(b)] for any bounded open set $V \subseteq X$ and $\epsilon>0$, there exists a compact set $C \subset V$ such that $\M_{V \setminus C}(T)<\epsilon$.
\end{enumerate}
For such $T$, there is a natural Borel regular outer measure $\|T\|$ on $X$, with $\|T\|(V) = \M_V(T)$ on bounded open sets $V$. An Ambrosio--Kirchheim current naturally determines a metric current with locally finite mass, and the mass measures of these objects agree on Borel sets (see \cite[section 2.5]{LW}). For a metric current with locally finite mass, we define $\set(T)$ in exactly the same way as  in \eqref{set_T}.

The concepts of boundary, push-forward, restriction, and slicing have natural extensions to the locally finite mass setting. Lang and Wenger define \emph{locally normal currents}, \emph{locally integer rectifiable currents}, and \emph{locally integral currents}; again we refer the reader to \cite{LW} for details.

We now give:

\begin{definition}
A \emph{local integral current space} of dimension $m \geq 0$ is a triple $N=(X,d,T)$ in which $(X,d)$ is a metric space and $T$ is a locally integral $m$-current on $(\bar X, \bar d)$ with $\set(T)=X$.
\end{definition}

An important example of a local integral current space is a complete, connected, oriented $C^0$ Riemannian manifold $(M,g)$, with the canonical choice of $T$ as in \eqref{eqn_canonical}, defined on the appropriate class of functions (see \cite[section 2.8]{LW}). 

Fortunately, there is a direct relationship between local integral current spaces and integral current spaces. We work with the metric space completion for consistency with \cite{AK} (although see the discussion at the end of \cite[section 2.2]{LW}).
\begin{lemma}
\label{lemma_local_IC}
Let $(X,d)$ be a metric space, and let $q \in X$. Let $T$ be a locally integral $m$-current on $(\bar X, \bar d)$, $m \geq 1$. Then for almost every $r>0$,
$T \llcorner B(q,r)$, defined naturally on $\mathcal{D}^m$, is an integral $m$-current on $(\bar X, \bar d)$.

In particular, if $N=(X,d,T)$ is a local integral current space of dimension $m\geq0$, then
$$N \llcorner B(q,r) := (\set(T \llcorner B(q,r)), \bar d, T \llcorner B(q,r))$$
is an integral current space of dimension $m$ for almost every $r>0$.
\end{lemma}
This is essentially \cite[ Lemma 2.34]{S14}, extended to the locally finite mass setting. 

In the statement and proof of this lemma, $B(q,r)$ denotes an open ball in $(\bar X, \bar d)$.
\begin{proof}
Define $T_r: \mathcal{D}^m \to \R$ (where $\mathcal{D}^m$ denotes Lipschitz $(m+1)$-tuples on $\bar X$ with the first argument bounded) by
$$T_r(f, \pi_1, \ldots, \pi_m) = T(f\chi_{B(q,r)}, \pi_1, \ldots, \pi_m).$$
This is a well-defined multilinear functional since $f\chi_{B(q,r)}$ is a bounded Borel function with bounded support (cf. \cite[section 2.3]{LW}). Next, note $\|T\|(B(q,r))$ is finite by the definition of metric current with locally finite mass.  As pointed out in \cite[section 2.5]{LW}, $T_r$ is a current on $(\bar X, \bar d)$ (since $\chi_{B(q,r)}$ is a bounded Borel function of bounded support). Moreover, $T_r$ is integer rectifiable (in the sense of Ambrosio--Kirchheim), by the results in \cite[section 2.7]{LW} (specifically in the proof of Theorem 2.2).
Thus, it suffices to show $\partial T_r$ has finite (Ambrosio--Kirchheim) mass for almost every $r>0$. 

Again from  \cite[section 2.5]{LW} the (Ambrosio--Kirchheim) mass of  $\partial T_r$ agrees with the mass of $\partial(T \llcorner B(q,r))$. By \cite[Theorem 2.1]{LW} (the slicing theorem for locally normal currents) that for almost every $r>0$, $\partial(T \llcorner B(q,r))$ has finite mass for almost every $r>0$ (cf. the proof of \cite[ Lemma 2.34]{S14} for further details). This completes the proof.
\end{proof}

We now adapt the notion of local integral current space to the asymptotically flat setting. While several definitions here are possible, we use the following.

\begin{definition}
\label{def_AF_LICS}
An $m$-dimensional ($m \geq 3$) \emph{asymptotically flat} (AF) local integral current space $N=(X,d,g,T)$ is quadruple such that $(X,d,T)$ is a connected local integral current space of dimension $m$ with $\partial T=0$, and:
\begin{enumerate}
\item[(a)] There exists a compact set $K \subset X$ such that $\overline{X \setminus K}$ (the closure of $X \setminus K$ in $X$ with respect to $d$) is a smooth $m$-manifold-with-boundary, and  $(\overline{X \setminus K},g)$ is a $C^0$ AF end.
\item[(b)] The metric $d$ restricted to $\overline{X \setminus K}$ is locally compatible with the Riemannian metric $g$ on $\overline{X \setminus K}$ (see Definition \ref{def_locally_compatible} below).
\item[(c)] The restriction of $T$ to $\overline{X \setminus K}$ is the canonical integral current \eqref{eqn_canonical}.

\end{enumerate}
\end{definition}

Note that the metric $d_g$ on $\overline{X \setminus K}$ depends on $K$, but we will always assume $K$ is fixed. Note that Huisken's isoperimetric mass of $N$, for $m=3$, is well-defined, independent of $K$, for such a space (as in Definition \ref{def:isop_mass}), explained as follows. First, the ``volume'' of $\Omega\subset X$ is defined by the $g$-volume of $\Omega \cap (X \setminus K)$ plus $\|T\|(\Omega \cap K)$, which will equal $\|T\|(\Omega)$, by Lemma \ref{lemma_mass_measure}(a). Second, the perimeter of $\Omega \subset X$ is well-defined for $\Omega$ containing $K$, since $\partial \Omega$ lies in $X \setminus K$; moreover, in the definition of isoperimetric mass, we may restrict to open sets that contain $K$.

Note also that an AF local integral current space is automatically locally compact.

An important example of an AF local integral current space is constructed from an oriented $C^0$ AF Riemannian manifold $(M,g)$ by taking the metric $d_g$ and the canonical locally integral current $T$ on $M$ given by integration. The purpose of Definition \ref{def_AF_LICS} is to allow even more general spaces as possible limits in Theorem \ref{thm1}.

\subsection{Local compatibility of distance functions and Riemannian metrics}
\begin{definition}
\label{def_locally_compatible}
Let $(M,g)$ be a connected $C^0$ Riemannian manifold, possibly with boundary, and let $d$ be a metric (distance function) on $M$. Let $d_g$ be the Riemannian distance function on $M$ with respect to $g$. We say $g$ and $d$ are \emph{locally compatible} if given any $p \in M$, there exists a $d$-open set $U$ in $M$ containing $p$ so that for all $x,y \in U$, $d(x,y) = d_g(x,y)$. Such $U$ will be called a \emph{neighborhood of compatibility} of $d$ and $g$.
\end{definition}

For example, if $M$ is $\R^n$ minus a ball, $g$ is the restriction of the Euclidean Riemannian metric to $M$, and $d$ is the restriction of the Riemannian distance function, then $g$ and $d$ are locally compatible though $d_g$ and $d$ are unequal.

In general, it is straightforward to verify that if $g$ and $d$ are locally compatible, then
\begin{equation}
\label{eqn_d_d_g}
d \leq d_g
\end{equation}
as functions on $M \times M$.

\begin{lemma}
	In the above definition, $d$ and $d_g$ induce the same topology, which agrees with the manifold topology.
\end{lemma}

In this proof we'll use $B_d$ and $B_g$ to denote open balls with respect to $d$ and $d_g$, respectively.

\begin{proof}
	First, it is clear that $d_g$ induces the manifold topology, since $M$ admits smooth Riemannian metric $g'$ such that $d_g$ and $d_{g'}$ are uniformly equivalent.

	We show $d$ and $d_g$ induce the same topology. First, let $U \subseteq M$ be $d$-open and nonempty, and let $p \in U$. Let $r>0$ be sufficiently small so that $B_d(p,r)\subseteq U$. Then by \eqref{eqn_d_d_g},  $B_g(p,r) \subseteq B_d(p,r) \subseteq U$. This shows one direction.
	
	Second, let $U \subseteq M$ be $d_g$-open, and let $p \in U$. There exists $r>0$ such that $B_g(p,r) \subseteq U$. By the definition of locally compatible there exists $r'>0$ such that $B_d(p,r')$ is a neighborhood of compatibility of $g$ and $d$. Without loss of generality, we may take $r' \leq r$.  Now if $q \in B_d(p,r')$, then $r'> d(p,q) = d_g(p,q)$, so that $q \in B_g(p,r')$. Then
	$$B_d(p,r') \subseteq B_g(p,r') \subseteq B_g(p,r) \subseteq U,$$
	proving the other direction.
\end{proof}

\subsection{(Pointed) intrinsic flat convergence}
Wenger defined the flat distance \cite{Wen} between two integral $m$-currents $T_1, T_2$ on a complete metric space $(X,d)$ to be
$$d_X^F(T_1,T_2) = \inf_{A,B} \left\{ \M(A) + \M(B) \; \big| \; T_2 - T_1 = A+ \partial B\right\},$$
where $A$ and $B$ are integral currents on $X$ of dimension $m$ and $m+1$, respectively. This is a direct generalization of Whitney's notion of the flat distance between submanifolds \cite{Whi}, which was extended to integral currents by Federer and Fleming \cite{FF}.

Inspired by the Gromov--Hausdorff distance and the flat distance, Sormani and Wenger defined the intrinsic flat distance as follows.

\begin{definition}[\cite{SW}]
Suppose $N_1=(X_1, d_1, T_1)$ and $N_2=(X_2, d_2, T_2)$ are integral current spaces of dimension $m$. The \emph{intrinsic flat distance} between $N_1$ and $N_2$ is defined to be
$$d_\F(M_1, M_2) =\inf \left\{ d_Z^F\left(\varphi_{1\#} (T_1), \varphi_{2\#}(T_2)\right)\right\},$$
where the infimum is taken over all complete metric spaces $(Z,d_Z)$ and isometric embeddings $\varphi_1:X_1 \to Z$, $\varphi_2:X_2 \to Z$.
\end{definition}

It was proved in \cite{SW} that $d_\F$ defines a metric on the space of precompact integral current spaces, modulo current-preserving isometries. If a sequence of (precompact) integral current spaces $N_j$ converges in $d_{\F}$ to an integral current space $N$ as $j \to \infty$, we write $N_j \xrightarrow{\F} N$.

The following result is very useful, allowing one to consider a single fixed ``common space'' in an $\F$-limit.

\begin{thm} [Theorem 4.2 of \cite{SW}]
\label{thm_common_space}
If $N_j = (X_j, d_j, T_j)  \xrightarrow{\F} N=(X,d,T)$ as integral current spaces, then there exists a complete metric space $(Z,d_Z)$ and isometric embeddings
$\varphi_j :(X_j, d_j) \to (Z,d_Z)$ and $\varphi:(X,d) \to (Z,d_Z)$ such that $\varphi_{j\#} T_j \to \varphi_{\#} T$ in $d_Z^F$.
\end{thm}

Note that the embeddings $\varphi_j$ and $\varphi$ extend uniquely to embeddings of their completions $\bar X_j$ and $\bar X$; we use the same notation for these maps and their extensions.

If $N_j = (X_j, d_j, T_j)  \xrightarrow{\F} N=(X,d,T)$ and $x_j \in X_j$, Sormani defines \emph{convergence of points} $x_j \to x \in (\bar X, \bar d)$ as the existence of isometric embeddings and $(Z, d_Z)$ as in Theorem \ref{thm_common_space} for which 
\begin{equation}
\label{eqn_points_converge}
\varphi_j(x_j) \to \varphi(x) \text{ in } Z.
\end{equation}

It was also proven in \cite{SW} that if $N_j \xrightarrow{\F} N$, then
\begin{equation}
\label{eqn_M_LSC}
\liminf_{j \to \infty} \M(N_j) \geq \M(N),
\end{equation}
i.e. the mass can only converge or drop in an intrinsic flat limit. (This statement should not be confused with the lower semicontinuity of ``mass'' of another sort in Theorem \ref{thm1}!) It will be important for us to study a stronger type of convergence in which the mass does not drop, also used by Portegies \cite{Por}. Sormani \cite{S16} makes the following convenient definition:

\begin{definition}[\cite{S16}]
Suppose $N_1=(X_1, d_1, T_1)$ and $N_2=(X_2, d_2, T_2)$ are precompact integral current spaces of dimension $m$. The \emph{intrinsic flat volume distance} between $N_1$ and $N_2$ is defined to be
$$d_{\VF}(N_1, N_2) =d_\F(N_1, N_2) + \left|\M(N_2) - \M(N_1)\right|.$$
\end{definition}
Thus, $N_j \xrightarrow{\VF} N$ if and only if $N_j \xrightarrow{\F} N$ and $\M(N_j) \to \M(N)$.

We now define pointed intrinsic flat convergence for local integral current spaces, which is tantamount to $\F$-convergence on balls. See Remark \ref{rmk_pointed} for a comparison with other closely related definitions.

\begin{definition}
\label{def_pointed}
A sequence $N_j=(X_j,d_j,T_j,p_j)$ of complete, pointed local integral current spaces of dimension $m$ converges to a complete, pointed local integral current space $N=(X,d,T,q)$ of dimension $m$ in the \emph{pointed intrinsic flat sense} (respectively, \emph{pointed intrinsic flat volume sense}) if
for any $r_0 > 0$, there exists $r \geq r_0$ such that $N \llcorner B(q,r)$ and $N \llcorner B(p_j,r)$ are precompact integral current spaces (for all $j$ sufficiently large), and $N_j \llcorner B(p_j,r) \xrightarrow{\F} N \llcorner B(q,r)$ (respectively, $N_j \llcorner B(p_j,r) \xrightarrow{\VF} N \llcorner B(q,r)$), and if $p_j \to q$ as in \eqref{eqn_points_converge}.
\end{definition}

Now, let $(M_j, g_j)$ be a sequence of smooth, connected, oriented, complete Riemannian $m$-manifolds (without boundary), inducing metrics $d_{g_j}$ and locally integral $m$-currents  $T_j$ as in \eqref{eqn_canonical}, and let $N_j = (M_j,d_{g_j},T_j)$ be the corresponding local integral current spaces.
Given $p_j \in M_j$ we have that $N_j \llcorner B(p_j,r)$
is a precompact integral current space for almost all $r>0$ by Lemma \ref{lemma_local_IC}. (Although not needed, since $M_j$ is a manifold, this actually holds for all $r>0$ by \cite[Lemma 2.35]{S14}.) In this way, it makes sense to say (abusing notation slightly) that $(M_j, g_j, p_j)$ converges in the pointed $\F$ or $\VF$ sense to an AF local integral current space $N=(X,d,T,q)$.

At this point, we have defined everything in the statement of Theorem \ref{thm1}.

\begin{remark}
\label{rmk_pointed}
Definition \ref{def_pointed} is closely related to convergence in the local flat topology defined by Lang and Wenger \cite{LW} (an extrinsic notion) and convergence in the pointed intrinsic flat distance between local integral current spaces defined by S. Takeuchi \cite[Definition 3.1]{Tak}. Suppose $N_j=(X_j,d_j,T_j,p_j)$ converges to $N=(X,d,T,q)$ (all complete pointed local integral current spaces of the same dimension) in the sense guaranteed by Takeuchi's compactness result \cite[Theorem 1.1]{Tak}. (This theorem is a reinterpretation of Lang and Wenger's extrinsic compactness result, \cite[Theorem 1.1]{LW} via Takeuchi's definition.) Then it is straightforward to check that convergence in the sense of Definition \ref{def_pointed} holds (using \cite[Proposition 3.6]{Tak}, the definition of convergence in the local flat topology, and \eqref{eqn_restrict_pushforward2}). 
%(To elaborate on these ``standard arguments'': Prop 3.6 of [Tak]  gives the existence of a complete LICS $(Z,d_Z, T_Z,)$ and appropriate isometric embeddings $\varphi_j, \varphi$, such that $\varphi_{\#j}T_j$ converges to $\varphi_{\#} T = T_Z$ in the local flat topology, and $\varphi_j(p_j) \to \varphi(q)$. This implies that for all $r > 0$ $\varphi_{\#j}T_j - \varphi_{\#} T$ equals $R_j + \partial S_j$, where $R_j$ and $S_j$ are local integral currents on $Z$, and $\|R_j\|(B(\varphi(q),r)$ and $\|S_j\|(B(\varphi(q),r)$ both converge to 0 as $j \to \infty$.
\end{remark}

\section{Volume, mass measure, and perimeter}
\label{sec_measures}
The purpose of this section to reconcile some Riemannian concepts with their integral current analogs. The first-time reader may prefer to read only the statements of Lemmas \ref{lemma_mass_measure} and \ref{lemma_perimeter_bdry_mass} before proceeding to section \ref{sec_regions}.

We first show that on a $C^0$ Riemannian manifold, the Lebesgue measure agrees with the mass measure $\|T\|$ associated to  the canonical integral current $T$ given by integration, i.e. \eqref{eqn_canonical}. We also give sufficient conditions on a function $u$ so that $\|T \llcorner du\| = \|T\|$ (which will be useful in later calculations involving the Ambrosio--Kirchheim slicing theorem).

\begin{lemma}
\label{lemma_mass_measure}
Suppose $(M,g)$ is a connected, oriented $C^0$ Riemannian manifold, possibly with boundary, inducing Lebesgue measure $\mu_g$. Suppose $(M,d,T)$ is an integral current space such that $T$ is given by (\ref{eqn_canonical}). Suppose that $g$ and $d$ are locally compatible on $M$ (see Definition \ref{def_locally_compatible}).
\begin{enumerate}
\item[(a)]  $\|T\|=\mu_g$ as Borel measures on $M$. In particular, $\M(T) = |M|_g$.
\item[(b)] There exists a universal $\gamma_0>0$  so that if $\gamma \in [0, \gamma_0]$, and if $u$ is any $(1+\gamma)$-Lipschitz function on $(M,d)$ that is $C^1$ except on a closed set $S\subset M$ of measure zero, with $(1+\gamma)^{-1}\leq |\nabla u|_g \leq (1+\gamma)$ on $M \setminus S$, then 
$$(1+\gamma)^{-1} \mu_g \leq \|T \llcorner du\| \leq (1+\gamma) \mu_g$$ 
as Borel measures on $M$. In particular, if $\gamma=0$, then $\|T \rst du\| = \mu_g$.
\end{enumerate}
\end{lemma}

An analogous result holds if $(M,d,T)$ is a local integral current space, but we will not need this.

\begin{remark}
Lemma \ref{lemma_mass_measure}(a) is known for smooth Riemannian manifolds (see \cite[Example 2.32]{SW}), from the general results in \cite{AK}. We include a direct proof, which also explicitly accounts for the metric possibly being only $C^0$.
\end{remark}

\begin{proof}
From the definition of local compatibility and the triangle inequality, we see
$$d_g(x,y) \geq d(x,y)$$
for all $x, y \in M$. 
In particular, for any function $\phi:M \to \R$, 
$$\Lip_{d_g} (\phi) \leq \Lip_d(\phi).$$
In the proof below,  ``Lipschitz function'' will refer to a Lipschitz function with respect to $d$ (and hence with respect to $d_g$)

To prove (a), recall that $\|T\|$ is defined to be the smallest Borel measure $\nu$ for which
$$|T(f,\pi_1, \ldots, \pi_m)| \leq \int_M |f| d\nu$$
for all bounded Lipschitz $f$ and all Lipschitz $\pi_i$ with $\Lip_d(\pi_i) \leq 1$. Note that
\begin{align*}
T(f,\pi_1, \ldots, \pi_m) &= \int_M f d\pi_1 \wedge \ldots \wedge d\pi_m\\
&= \int_M f h \Omega_g,
\end{align*}
where $\Omega_g$ is the oriented volume form on $(M,g)$ and $h(x)$ is defined a.e. so that
$$d\pi_1 \wedge \ldots \wedge d\pi_m = h(x) \Omega_g$$
a.e. as elements of $\Lambda^m(T^*M)$.  (Here and below, ``a.e.'' is with respect to $\mu_g$.) Since  $\Lip_{d_g}(\pi_i) \leq \Lip_d(\pi_i) \leq 1$, we have $|d\pi_i|_g \leq 1$ a.e., and hence
 $|h(x)| \leq 1$ a.e. Then
$$\left|T(f,\pi_1, \ldots, \pi_m)\right| \leq \int_M |f| d\mu_g.$$
This implies that $\|T\| \leq \mu_g$, showing one direction.

From $\|T\| \leq \mu_g$, it follows that $\|T\|$ is absolutely continuous with respect to $\mu_g$ as Borel measures; by the Radon--Nikodym Theorem, there exists a nonnegative Borel function $\rho:M \to \R$ such that
\begin{equation}
\|T\|(E) = \int_E \rho d\mu_g \label{eqn_RN}
\end{equation}
for all Borel sets $E \subseteq M$. By the first part of the proof, $\rho \leq 1$ a.e.

Let $1-a$ be the essential infimum of $\rho$ on $M$ (with respect to $\mu_g$), where $a \in [0,1]$. We claim that $a=0$, which implies $\|T\| = \mu_g$. If $a>0$, the set $A$ on which $\rho \leq 1-\frac{a}{2}$ has positive measure with respect to $\mu_g$. By Lebesgue's density theorem, $A$ has density equal to 1 at almost every point of $A$ (with respect to $\mu_g$). Choose a point $p \in A$ with density 1 and choose some $r_0 > 0$ such that
$$\frac{\mu_g(A \cap B(p,r))}{\mu_g(B(p,r))} \geq 1-\frac{a}{6},$$
for all $r \in (0,r_0]$. In particular, for such $r$,
\begin{equation}
\label{eqn_density}
\frac{\mu_g(B(p,r) \setminus A)}{\mu_g(B(p,r))} \leq \frac{a}{6}.
\end{equation}
Let $f_r$ be the characteristic function of $B(p,r)$, a bounded Borel function.

For any choice of $\pi_1, \ldots, \pi_m$ with $\Lip_d(\pi_i) \leq 1$, and $r \in (0,r_0]$,
\begin{align*}
|T(f_r,\pi_1, \ldots, \pi_m)| &\leq \int_{M} f_r d \|T\|\\
&= \int_{B(p,r)} \rho d \mu_g\\
&= \int_{B(p,r) \cap A} \rho d \mu_g+\int_{B(p,r) \setminus A} \rho d \mu_g\\ 
&\leq \left(1-\frac{a}{2}\right) \mu_g(B(p,r) \cap A) + \mu_g(B(p,r) \setminus A)\\
&\leq \left(1-\frac{a}{2}\right) \mu_g(B(p,r)) + \frac{a}{6} \cdot \mu_g(B(p,r)) \\
&=\left(1-\frac{a}{3}\right)\mu_g(B(p,r)),
\end{align*}
having used the definition of $\|T\|$, \eqref{eqn_RN}, \eqref{eqn_density}, and the definition of $A$. 
In particular, for any Lipschitz $\pi_1, \ldots, \pi_m$ and $r \in (0,r_0]$,
\begin{equation}
|T(f_r,\pi_1, \ldots, \pi_m) \leq \left(\prod_{i=1}^m \Lip_d(\pi_i)\right) \left(1-\frac{a}{3}\right)\mu_g(B(p,r)). \label{eqn_Tf}
\end{equation}

On the other hand, let $\omega_1, \ldots, \omega_m$ be an oriented $g$-orthonormal basis of $T_p^* M$ (where $p$ is the same point as chosen above). Then $\omega_1 \wedge \ldots \wedge \omega_m = \Omega_g(p)$. Now, choose $C^1$ functions $\pi_1, \ldots, \pi_m$ on a small neighborhood $W$ of $p$ such that $d\pi_i = \omega_i$ at $p$. Define
$$\beta = \left(\frac{1-\frac{a}{4}}{1-\frac{a}{3}}\right)^{\frac{1}{m}} > 1.$$
We can choose $W$ small enough so that $|d\pi_i|_g < \beta$
on $W$, so that $\Lip_{d_g}(\pi_i) < \beta$ on $W$ as well. By local compatibility, we can shrink $W$ if necessary to arrange that $\Lip_d(\pi_i)< \beta$ on $W$. Now, extend each $\pi_i$ to $M$ so that $\Lip_d(\pi_i)<\beta$ on $M$.  Now, 
$$d \pi_1 \wedge \ldots \wedge d\pi_m = h(x) \Omega_g$$
a.e. as measurable differential $m$-forms on $M$, where $h(x)$ is a function on $M$ (defined a.e.) that is continuous on $W$ (since $g$ is $C^0$ and $\pi_i|_W$ are $C^1$), and $h(p)=1$. Thus, we may choose $r \in (0,r_0]$ so that $h(x) \geq 1 - \frac{a}{4}$ on $B(p,r)$.
Then
\begin{align}
T(f_r,\pi_1, \ldots, \pi_m) &=  \int_{B(p,r)}  d\pi_1 \wedge \ldots \wedge d\pi_m\nonumber\\
&= \int_{B(p,r)} h \Omega_g\nonumber\\
&\geq \left(1-\frac{a}{4}\right)\mu_g(B(p,r)). \label{eqn_Tf2}
\end{align}
Inequalities \eqref{eqn_Tf} and \eqref{eqn_Tf2} contradict if $a>0$, by the choice of $\beta$, proving the first part of the lemma.

For (b), let $u$ be as in the hypothesis, for some $\gamma \geq 0$. Recall that $T \llcorner du$ is the $(m-1)$-current defined by
$$(T \llcorner du)(f,\pi_1, \ldots, \pi_{m-1}) = T(f,u,\pi_1, \ldots, \pi_{m-1}).$$
Since $\Lip_d(u) \leq 1+\gamma$, we have $\|T \llcorner du\| \leq (1+\gamma) \|T\|= (1+\gamma) \mu_g$ by (a). This shows one of the inequalities and also establishes, by the Radon--Nikodym theorem, that $\|T \llcorner du\|$ can be represented by integration of $\rho d\mu_g$ for some function $\rho \leq 1+\gamma$ a.e. In particular, we may neglect points $p \in S$ in the argument below.
For the other inequality,  a slight modification to the proof of (a) is needed. Let $(1+\gamma)^{-1}(1-a)$ be the essential infimum of $\rho$, for some number $a$, and suppose that $a>0$ (or else we are done). Since $a>0$, the set $A$ on which $\rho \leq (1+\gamma)^{-1}(1-\frac{a}{2})$ has positive measure with respect to $\mu_g$. Choose a point $p \in A \setminus S$ and a radius $r_0$ as in \eqref{eqn_density}, and again let $f_r$ be the characteristic function of $B(p,r)$.
 Then similarly to the chain of inequalities \eqref{eqn_Tf}, we have for any $\pi_1, \ldots \pi_{m-1}$ with $\Lip_d(\pi_i) \leq 1$, and any $r \in (0,r_0]$,
\begin{align}
(T \llcorner du)(f, \pi_1, \ldots, \pi_{m-1}) &\leq \int_{B(p,r)} d\|T \llcorner du\|\nonumber\\
&= \int_{B(p,r)} \rho d\mu_g\nonumber\\
&= \int_{B(p,r) \cap A} \rho d\mu_g + \int_{B(p,r) \setminus A} \rho d\mu_g\nonumber\\
&\leq (1+\gamma)^{-1} \left(1-\frac{a}{2}\right) \mu_g(B(p,r)) + \frac{a}{6} (1+\gamma)\mu_g(B(p,r)) \nonumber\\
&\leq (1+\gamma)^{-1} \left(1-\frac{a}{2} + \frac{a}{5}\right) \mu_g(B(p,r)), \label{eqn_Tf3}
\end{align}
provided $(1+\gamma)^2 \leq \frac{6}{5}$. Assume $\gamma \leq \sqrt{\frac{6}{5}}-1 =: \gamma_0$.

Now, take $\omega_1, \ldots, \omega_{m-1} \in T_p^*M$ so that $du(p), \omega_1, \ldots, \omega_{m-1}$ give an oriented $g$-orthogonal basis of $T_p^*M$, with $\omega_1, \ldots, \omega_{m-1}$ of unit length. Let $\pi_1, \ldots, \pi_{m-1}$ be Lipschitz functions on $M$ with $\Lip_d(\pi_i) < \beta$, such that $d\pi_i = \omega_i$ at $p$ and $\pi_1, \ldots, \pi_{m-1}$ are $C^1$ on a neighborhood of $p$ (as explained in (a)), except here we choose 
$$\beta =  \left(\frac{1-\frac{a}{4}}{1-\frac{3a}{10}}\right)^{\frac{1}{m-1}} > 1.$$
 In particular, 
$$du \wedge d \pi_1 \wedge \ldots \wedge d\pi_{m-1} = h(x) \Omega_g$$
a.e. as measurable differential $m$-forms on $M$, where $h(x)$ is continuous at $p$. By construction, $h(p) \geq (1+\gamma)^{-1}$. Thus, we may chose $r \in (0,r_0]$ so that $h(x) \geq (1+\gamma)^{-1}\left(1 - \frac{a}{4}\right)$ on $B(p,r)$. Then as in \eqref{eqn_Tf2}, we have
\begin{align}
(T\llcorner du)(f_r,\pi_1, \ldots, \pi_{m-1}) &= \int_{B(p,r)} du \wedge d\pi_1 \wedge \ldots \wedge d\pi_{m-1}\nonumber\\
 &\geq (1+\gamma)^{-1} \left(1-\frac{a}{4}\right) \mu_g(B(p,r)). \label{eqn_Tf4}
\end{align}
Inequalities  \eqref{eqn_Tf3} and \eqref{eqn_Tf4} contradict if $a>0$ by this choice of $\beta$. Thus $a \leq 0$, and $\rho \geq (1+\gamma)^{-1}$ a.e. In particular, $\|T \rst du\| \geq (1+\gamma)^{-1} \mu_g$, completing the proof of the lemma.
\end{proof}

Next, we state a result that shows the concepts of perimeter and boundary mass are compatible in a $C^0$ Riemannian manifold. This generalization of \cite[Theorem 3.7]{AK} (which is the Euclidean case) will be proved in the appendix, where we also recall the details of how the perimeter is defined in general (and in particular with respect to $C^0$ Riemannian metrics).

\begin{lemma}
\label{lemma_perimeter_bdry_mass}
Let $(M,g)$ be a connected, oriented $C^0$ Riemannian manifold of dimension $m$, possibly with boundary. Suppose $d$ is a complete metric on $M$, locally compatible with $g$, and let $E \subseteq M$ be a precompact Borel set. Let $T_E$ be the integer rectifiable $m$-current on $(M,d)$ given by integration over $E$. Then $\M(\partial T_E)$ is finite if and only if $E$ has finite perimeter with respect to $g$, and in this case, 
$$|\partial^* E|_g = \M(\partial T_E).$$
\end{lemma}

\section{Convergence of subregions of integral current spaces}
\label{sec_regions}

In this section, we consider precompact integral current spaces $N_j = (X_j,d_j,T_j)$ converging in $\F$ to $N=(X,d,T)$. We prove some general results regarding the $\F$-convergence of regions in $X_j$ to a region in $X$. Specifically, in section \ref{sublevel}, we show that almost all sub-level sets of a Lipschitz function on $X$ are themselves $\F$ limits of regions in $X_j$ (Lemma \ref{lemma_K_A_converge}). In section \ref{balls}, we show that balls and annuli in $X_j$ converge to corresponding balls and annuli in $X$ for almost all radii, provided the base points converge (Lemma \ref{lemma_convergence_balls_annuli}). Lastly, in section \ref{sec_converge_perim}, we apply these results to show that we have (or nearly have) convergence of the boundary masses of many of these regions, under suitable conditions (Propositions \ref{prop_perimeters_converge} and \ref{prop_perimeters_converge2}). These results will be used later in the proof of Theorem \ref{thm1}, and they may have applications to other problems on $\F$-convergence.

\subsection{Convergence of regions defined as sub-level sets}
\label{sublevel}
Fix a Lipschitz function $u:X \to \R$, and define, for any $\delta \in \R$
\begin{equation}
\label{eqn_E_delta}
E^\delta = \{u \leq \delta\} \subseteq X.
\end{equation}
Our goal here is to construct a ``corresponding region'' $E^\delta_j$ in $X_j$ for each $j$ such that $N_j \llcorner E^\delta_j \toF N \llcorner E^\delta$ as $j \to \infty$.

We begin with:
\begin{lemma}
\label{lemma_K_delta}
For almost all $\delta \in \R$,  $T \llcorner E^\delta$ is an integral current on $(\bar X, \bar d)$, and
$$\{u < \delta\} \subseteq \set(T \llcorner E^\delta) \subseteq E^\delta.$$
In particular, $N \rst E^\delta$ is an integral current space for such $\delta$.
\end{lemma}
The statement and proof are direct generalizations of \cite[Lemma 2.34]{S14}, with very minor modifications.
\begin{proof}
First, it is clear that $T \llcorner E^\delta$ is an integer rectifiable current for all $\delta \in \R$, so we only need to show its boundary has finite mass.
From the definition of slice,
\begin{align*}
\partial \left(T \llcorner E^{\delta}\right) &= \langle T ,u,\delta \rangle +(\partial T ) \llcorner u^{-1}(-\infty,\delta].
\end{align*}
But $\M((\partial T) \llcorner u^{-1}(-\infty,\delta]) \leq \M(\partial T) < \infty$ for all $\delta$. By Theorem \ref{thm_slicing},
$$\int_{-\infty}^\infty \M\langle T,u,s\rangle ds \leq \Lip(u)\M(T) < \infty.$$
In particular, $\M\langle T,u,s\rangle$ is finite for almost all $s$. It follows that $\partial \left(T \llcorner E^{\delta}\right)$ has finite mass, and hence $T \llcorner E^\delta$ is an integral current,  for almost all $\delta$.

To prove the second claim, if $x \in \{u < \delta\}$, then by continuity of $u$ there is a ball about $x$ contained in $\{u < \delta\}$, and hence in $E^\delta$. In particular, $\|T \llcorner E^\delta\|$ has the same lower density as $\|T\|$ at $x$, which is positive since $(X,d,T)$ is an integral current space. A similar argument shows the other inclusion.
\end{proof}

For real numbers $\delta_1 < \delta_2$, define the ``annular'' region
\begin{equation}
\label{eqn_A_delta}
A^{\delta_1,\delta_2} = \{\delta_1 \leq u \leq \delta_2\}.
\end{equation}

\begin{lemma}
\label{lemma_annulus}
For almost all $\delta_1 < \delta_2$,  $T \llcorner A^{\delta_1, \delta_2}$ is an integral current on $(\bar X, \bar d)$ and
$$\{ \delta_1 < u < \delta_2\} \subseteq \set\left(T \llcorner A^{\delta_1,\delta_2}\right) \subseteq A^{\delta_1,\delta_2}.$$
In particular, $N \rst A^{\delta_1,\delta_2}$ is an integral current space for such $(\delta_1,\delta_2)$.
\end{lemma}
\begin{proof}
Since
$$T \llcorner \{ u \leq \delta_2\} = T \llcorner \{ u <  \delta_1\} + T \llcorner A^{\delta_1, \delta_2},$$
it follows that
\begin{align*}
\partial \left(T \llcorner A^{\delta_1,\delta_2}\right) &= \partial \left( T \llcorner u^{-1}(-\infty, \delta_2]\right) -  \partial \left( T \llcorner u^{-1}(-\infty, \delta_1)\right)\\
&= \langle T, u, \delta_2\rangle + (\partial T) \llcorner u^{-1}(-\infty, \delta_2] - \langle T, - u, -\delta_1 \rangle - (\partial T) \llcorner u^{-1}(-\infty, \delta_1).
\end{align*}
The mass of $(\partial T) \llcorner u^{-1}(-\infty, \delta_2] -  (\partial T) \llcorner u^{-1}(-\infty, \delta_1)$ is bounded above by $\M(\partial T) < \infty$, and in the proof of the Lemma  \ref{lemma_K_delta} it was was shown that the slices $\langle T, u, s\rangle$ have finite mass for almost all $s$. The first  claim follows, and the second is a straightforward adaptation of the corresponding argument in the proof of Lemma \ref{lemma_K_delta}.
\end{proof}

We now proceed to construct the regions $E^\delta_j$ in $X_j$ that will $\F$-converge to $E^\delta$ for almost all $\delta$ (and the regions $A^{\delta_1,\delta_2}_j$ that will $\F$-converge to $A^{\delta_1, \delta_2}$ for almost all $\delta_1 <\delta_2$).
By Theorem \ref{thm_common_space}, there exists a complete metric space $(Z,d_Z)$ 
and isometric embeddings $\varphi_j$ of $(X_j,d_j)$ into $Z$ and $\varphi$ of $(X,d)$ into $Z$, such that
the pushed-forward currents converge in the flat sense in $Z$:
\begin{equation}
\label{eqn_dZF}
\varphi_{j\#} (T_j) \to \varphi_{\#} (T) \text{ in } d_Z^F.
\end{equation}
Let $U:Z \to \R$ be a Lipschitz extension of $u \circ \varphi^{-1}:\varphi(X)\subset Z \to \R$, with $\Lip(U) = \Lip(u)$. To be concrete, define
\begin{equation}
\label{eqn_U}
U(z) = \inf_{y \in \varphi(X)} \left\{u \circ \varphi^{-1}(y)+\Lip(u)d_Z(z,y)\right\}.
\end{equation}
Let 
\begin{equation}
\label{eqn_u_j}
u_j = U \circ \varphi_j,
\end{equation}
a Lipschitz function on $(X_j,d_j)$ with $\Lip(u_j) \leq \Lip(u)$. 
As in \eqref{eqn_E_delta} and \eqref{eqn_A_delta}, define
\begin{equation}
\label{eqn_E_j_delta}
E^\delta_j = \{u_j \leq \delta\} \subseteq X_j
\end{equation}
for $\delta\in \R$, and
$$A_j^{\delta_1,\delta_2} = \{\delta_1 \leq u_j \leq \delta_2\} \subseteq X_j.$$
for $\delta_1 < \delta_2$. 

The following is an immediate consequence of Lemmas~\ref{lemma_K_delta} and~\ref{lemma_annulus}.
\begin{cor} 
$\;$
\begin{enumerate}
\item[(a)] For almost all $\delta \in \R$, $T_j \llcorner E^\delta_j$ is an integral current on $(\bar X_j, \bar d_j)$ for all $j$. Moreover,
$$\{u_j < \delta\} \subseteq \set(T_j \llcorner E^\delta_j) \subseteq E^\delta_j.$$
In particular, for such $\delta$, $N_j \rst E^\delta_j$ is an integral current space for all $j$.
\item[(b)] For almost all pairs $(\delta_1, \delta_2)$ with  $\delta_1 < \delta_2$, $T_j \llcorner A_j^{\delta_1, \delta_2}$ is an integral current on $(\bar X_j, \bar d_j)$ for all $j$.
Moreover,
$$\{\delta_2 < u_j < \delta_1\} \subseteq \set(T_j \llcorner A_j^{\delta_1,\delta_2}) \subseteq A_j^{\delta_1,\delta_2}.$$
In particular, for such $(\delta_1, \delta_2)$, $N_j \rst A_j^{\delta_1,\delta_2}$ is an integral current space for all $j$.
\end{enumerate}
\end{cor}

%\begin{proof}
%Let $P_j \subset \R$ be the set of values $\delta$ for which $T_j \llcorner E^\delta_j$ is not an integral current. By Lemma \ref{lemma_K_delta}, $P_j$ has measure zero, and hence $\cup_j P_j$ has measure zero. This proves the first claim in (a), and the second follows from Lemma \ref{lemma_K_delta}. The proof of (b) follows very similarly.
%\end{proof}

\begin{lemma} 
\label{lemma_K_A_converge}
Upon passing to a subsequence:
$\;$
\begin{enumerate}
\item [(a)] For almost all $\delta$, $\varphi_{j\#}(T_j \llcorner E^\delta_j)$ converges in the flat sense in $Z$ to $\varphi_{\#}(T \llcorner E^\delta)$ as $j \to \infty$. In particular, $N_j \llcorner E^\delta_j \toF N \llcorner E^\delta$.
\item [(b)] For almost all $ \delta_1 < \delta_2$,  $\varphi_{j\#}(T_j \llcorner A^{\delta_1,\delta_2}_j)$ converges in the flat sense in $Z$ to $\varphi_{\#}(T \llcorner A^{\delta_1,\delta_2})$  as $j \to \infty$.  In particular, $N_j \llcorner A^{\delta_1,\delta_2}_j \toF N \llcorner A^{\delta_1,\delta_2}$.
\end{enumerate}
\end{lemma}
\noindent Note the subsequence does not dependent on $\delta$ (or $\delta_1$ and $\delta_2$).

\begin{proof}
We prove (a) by mimicking the proof  of \cite[Lemma 4.1]{S14}. By \eqref{eqn_dZF} and the definition of flat convergence, there exist integral currents $B_j$ and $C_j$ on $Z$ so that
\begin{equation}
\label{eqn_flat_converge}
\varphi_{j\#} (T_j ) - \varphi_{\#} (T) = \partial B_j + C_j
\end{equation}
and $\mathbb{M}( B_j)+\mathbb{M}(C_j) \to 0$. Next, from \eqref{eqn_restrict_pushforward2},
\begin{align*}
\varphi_{j\#} (T_j) \llcorner U^{-1}(-\infty, \delta] &=\varphi_{j\#}(T_j \llcorner E^\delta_j), \text{ and}\\
\varphi_{\#} (T) \llcorner U^{-1}(-\infty, \delta] &= \varphi_{\#}(T \llcorner E^\delta),
\end{align*}
for any $\delta \in \R$. Consider now $\delta \in \R$ for which $T \llcorner E^\delta$ and $T_j \llcorner E^\delta_j$ are integral currents for all $j$ (which holds for almost all $\delta$ by Lemma \ref{lemma_K_A_converge}). Take the difference and use \eqref{eqn_flat_converge}:
\begin{align*}
\varphi_{j\#}(T_j \llcorner E^\delta_j) - \varphi_{\#}(T \llcorner E^\delta) &=\left(\varphi_{j\#} (T_j) - \varphi_{\#} (T) \right) \llcorner U^{-1}(-\infty, \delta]\\ 
&=  (\partial B_j)  \llcorner U^{-1}(-\infty, \delta]  + C_j  \llcorner U^{-1}(-\infty, \delta]\\
&= -\langle B_j, U, \delta \rangle + \partial (B_j \llcorner U^{-1}(-\infty, \delta]) + C_j  \llcorner U^{-1}(-\infty, \delta].
\end{align*}
In particular, this provides an upper bound for the flat distance in $Z$:
$$d_Z^F (\varphi_{j\#}(T_j \llcorner E^\delta_j), \varphi_{\#}(T \llcorner E^\delta)) \leq f_j(\delta) + \mathbb{M}(B_j) + \mathbb{M}(C_j),$$
where
$$f_j(\delta) = \mathbb{M}\langle B_j, U, \delta \rangle \geq 0.$$
By Theorem \ref{thm_slicing} $\int_{-\infty}^\infty f_j(\delta)d\delta \leq \Lip(U)\mathbb{M}(B_j),$ which converges to zero as $j \to \infty$. Thus, there is a subsequence of $f_j(\delta)$ that converges to zero pointwise a.e. This shows (a).

For (b), the proof is similar to that of (a), except  $E^\delta_j$ and $E^\delta$ are replaced with $A_j^{\delta_1, \delta_2}$ and $A^{\delta_1, \delta_2}$, and  $U^{-1}(-\infty,\delta]$ is replaced with $U^{-1}[\delta_1, \delta_2]$.  
\end{proof}

Now, if we further assume $N_j \to N$ in $\VF$, then the masses of the regions subconverge as well:

\begin{lemma} Suppose $N_j \toVF N$. Upon passing to a subsequence:
\label{lemma_volumes}
\begin{enumerate} 
\item[(a)] for almost all $\delta \in \R$, $\M(T_j \llcorner E^\delta_j) \to \mathbb{M}(T \llcorner E^\delta)$ as $j \to \infty$, and 
\item[(b)] for almost all $ \delta_1 < \delta_2$, $\M(T_j \llcorner A_j^{\delta_1,\delta_2}) \to \mathbb{M}(T \llcorner A^{\delta_1,\delta_2})$ as $j \to \infty$.
\end{enumerate}
\end{lemma}
\begin{proof}
To show (a) we use (a) of Lemma \ref{lemma_K_A_converge}, and pass to a subsequence so that $\varphi_{j\#}(T_j \llcorner {E_j^{\delta}})$ converges in the flat sense in $Z$ to  $\varphi_{\#}(T \llcorner {E^{\delta}})$ for almost all $\delta$. Fix such a $\delta$. The rest of the proof follows from the more general result presented next, Lemma \ref{lemma_volumes_general}.

The argument for (b) is handled similarly, using (b) of Lemma \ref{lemma_K_A_converge}.
\end{proof}

\begin{lemma}
\label{lemma_volumes_general}
Suppose $N_j \xrightarrow{\VF} N$, and $Y_j \subseteq X_j$, $Y \subseteq X$ are Borel subsets so that $T_j \rst Y_j$ and $T \rst Y$ are integral currents on $\bar X_j$, and $\bar X$, respectively, and that  $\varphi_{j\#}(T_j \rst Y_j)$ converges to $\varphi_{\#}(T \rst Y)$ in $d_Z^F$ (where $\varphi_j$, $\varphi$, and $Z$ are as in \eqref{eqn_dZF}).
 Then
$$\lim_{j \to \infty} \M (T_j \rst Y_j) = \M(T \rst Y).$$
\end{lemma}
\begin{proof}
From \eqref{eqn_dZF} and the hypotheses,
$$\varphi_{j\#}(T_j) - \varphi_{j\#}(T_j \llcorner {Y_j}) \to \varphi_{\#}(T) - \varphi_{\#}(T \llcorner {Y}) \text{ in } d_Z^F.$$
It is easy to see that
$$\varphi_{j\#}(T_j ) - \varphi_{j\#}(T_j \llcorner {Y_j}) = \varphi_{j\#}(T_j \llcorner {\tilde Y_j}),$$
where ${\tilde Y_j}$ is the complement of ${Y_j}$ in $X_j$, and likewise 
$$\varphi_{\#}(T) - \varphi_{\#}(T \llcorner {Y}) = \varphi_{\#}(T \llcorner {\tilde Y}),$$
where ${\tilde Y}$ is the complement of ${Y}$ in $X$. In particular, the right-hand sides of the last two equations are integral currents, with $\varphi_{j\#}(T_j \llcorner {\tilde Y_j})$ converging to $\varphi_{\#}(T \llcorner {\tilde Y})$.

Using the $\VF$ hypothesis, the lower semicontinuity of $\M$ under flat convergence  \eqref{eqn_M_LSC}, and the additivity of $\|T_j\|$ and $\|T\|$ on disjoint Borel sets, we compute
\begin{align*}
\M(T) &=\lim_{j \to \infty} \M(T_j) \\
&= \lim_{j \to \infty} \left( \M(T_j \llcorner Y_j) + \M(T_j \llcorner{\tilde Y_j})\right)\\
&\geq \liminf_{j \to \infty} \M(T_j \llcorner Y_j) + \liminf_{j \to \infty} \M(T_j \llcorner{\tilde Y_j})\\
&\geq  \M(T \llcorner Y) + \M(T \llcorner{\tilde Y})\\
&= \M(T).
\end{align*}
Thus, equality holds throughout, and the claim follows.
\end{proof}

\subsection{Convergence of balls and annuli}
\label{balls}
Again, let $N=(X,d,T)$ be a precompact integral current space that is the $\F$-limit of precompact integral current spaces $N_j=(X_j,d_j,T_j)$, and suppose points $p_j \in X_j$ converge to $q \in X$ as in \eqref{eqn_points_converge}. Define the open metric annuli
$$A(p_j,r_1,r_2) =  B(p_j, r_2) \setminus  \bar B(p_j,r_1), \qquad A(q,r_1,r_2) =  B(q, r_2) \setminus  \bar B(q,r_1),$$
subsets of $X_j$ and $X$, respectively.

\begin{lemma}$\;$
\begin{enumerate}
\item [(a)] For almost all $r>0$, $T_j \llcorner B(p_j,r)$ and $T \llcorner B(q,r)$ are integral currents for all $j$, and 
\begin{align*}
B(p_j,r) &\subseteq \set(T_j \llcorner B(p_j,r)) \subseteq \bar B(p_j,r),\\
 B(q,r) &\subseteq \set(T \llcorner B(q,r)) \subseteq \bar B(q,r).
\end{align*}
\item [(b)] For almost all $r_2>r_1>0$, $T_j \llcorner A(p_j,r_1,r_2)$ and $T \llcorner A(q,r_1,r_2)$ are integral currents for all $j$, and 
\begin{align*}
A(p_j,r_1,r_2) &\subseteq \set(T_j \llcorner A(p_j,r_1,r_2)) \subseteq \bar A(p_j,r_1,r_2),\\ 
A(q,r_1,r_2) &\subseteq \set(T \llcorner A(q,r_1,r_2)) \subseteq \bar A(q,r_1,r_2).
\end{align*}
\end{enumerate}
\end{lemma}

\begin{proof}
Part (a) follows from \cite[Lemma 2.34]{S14}; (b) follows from a nearly identical argument.
\end{proof}

Using Theorem \ref{thm_common_space} again, let $\varphi_j$ and $\varphi$ be isometric embeddings of $X_j$ and $X$, respectively into some complete metric space $Z$, such that
\eqref{eqn_dZF}  and \eqref{eqn_points_converge} hold.

\begin{lemma} Upon passing to a subsequence:
\label{lemma_convergence_balls_annuli}
\begin{enumerate}
\item [(a)] For almost all $r>0$
$\varphi_{j\#}(T_j \llcorner B(p_j,r))$ converges in $d_Z^F$ to $\varphi_{\#}(T \llcorner B(q,r))$. In particular, $N_j \llcorner B(p_j,r) \toF N \llcorner B(q,r)$. 
\item [(b)] For almost all $r_1<r_2$
$\varphi_{j\#}(T_j \llcorner A(p_j,r_1,r_2))$ converges in $d_Z^F$ to $\varphi_{\#}(T \llcorner A(q,r_1,r_2))$. In particular, $N_j \llcorner A(p_j,r_1,r_2) \toF N \llcorner A(q,r_1,r_2)$. 
\end{enumerate}
\end{lemma}
\begin{proof}
Let $\rho_j(\cdot) = d_Z(\cdot,\varphi_j(p_j))$ and $\rho(\cdot) = d_Z(\cdot, \varphi(q))$, Lipschitz functions on $Z$. 
Part (a) is proved in the proof of \cite[Lemma 4.1]{S14}, applying the Ambrosio--Kirchheim slicing theorem to the functions $\rho_j$.
Part (b) is proved very similarly, replacing $\rho_j^{-1}(-\infty,r)$ in the proof with $\rho_j^{-1}(r_1, r_2)$ and $\rho^{-1}(-\infty,r)$ with $\rho^{-1}(r_1, r_2)$.
\end{proof}

Now, we give some results for the case of $\VF$ convergence that are immediate consequences of  Lemma \ref{lemma_volumes_general}.

\begin{lemma} 
If $N_j=(X_j,d_j,T_j) \toVF N=(X,d,T)$, then upon passing to a subsequence,
\begin{enumerate}
\item[(a)] for almost all $r>0$, $\M(T_j \llcorner B(p_j,r)) \to \M(T \llcorner B(q,r))$, and 
\item[(b)] for almost all $r_2>r_1>0$, $\M(T_j \llcorner A(p_j,r_1,r_2)) \to \M(T \llcorner A(q,r_1,r_2))$.
\end{enumerate}
\end{lemma}

\subsection{Convergence of boundary masses}
\label{sec_converge_perim}
In the previous two subsections, we assumed $N_j \toF N$ (or $N_j \toVF N$) and constructed regions in $N_j$ that $\F$- (or $\VF$-) converge to a specified region in $N$.
 In general, however, the masses of the \emph{boundaries} of the subregions will only be lower semicontinuous. (This follows from \eqref{eqn_M_LSC}, noting that if $N_j \toF N$, then $\partial N_j \toF \partial N$.)   The purpose of this section is to formulate hypotheses to obtain convergence (or approximate convergence) of the boundary masses. This is needed because a drop in boundary mass would lead to an unfavorable change in the isoperimetric mass in the proof of Theorem \ref{thm1}, as discussed in the introduction.

\smallskip

\paragraph{\emph{Setup:}}
Suppose $(M_j, g_j, p_j)$ is a sequence of smooth, pointed, oriented asymptotically flat $m$-manifolds (with corresponding locally integral currents $T_j$ defined by \eqref{eqn_canonical}, so these can be viewed as pointed AF local integral current spaces $N_j=(M_j, d_{g_j},g_j, T_j,p_j)$) that converges in the pointed $\VF$ sense  to a complete AF local integral current space $N=(X,d,g,T,q)$ of dimension $m$. 
Fix $K \subset X$ as in Definition \ref{def_AF_LICS}, and let $E \subset X$ be a compact set containing $K$, with $\partial E$ located in the open $C^0$ AF end $X \setminus K$, and $\partial E$ smooth. 
Choose $R>0$ sufficiently large so that $E \subset B(q,R-1)$ and $N_j \llcorner B(p_j,R)$ converges in $\VF$ to $N \llcorner B(q,R)$ and $p_j \to q$ (using Definition \ref{def_pointed}). See Figure \ref{fig_setup}.

\begin{figure}[ht]
	\begin{center}
		\includegraphics[scale=0.60]{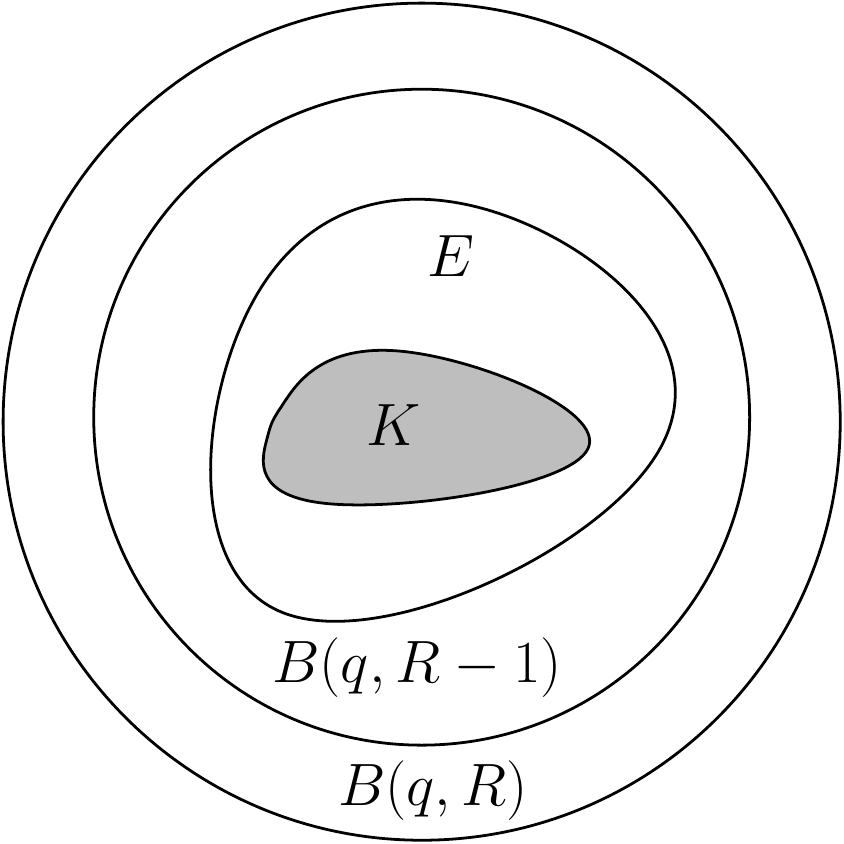}
		\caption{\small The metric space $X$ is shown, with some of the subsets used in the setup of section \ref{sec_converge_perim}.
			\label{fig_setup}}
	\end{center}
\end{figure}

By Theorem \ref{thm_common_space} there exists a complete metric space $(Z,d_Z)$ 
and isometric embeddings $\varphi_j$ of $\bar B(p_j,R)$ into $Z$ and $\varphi$ of $\bar B(q,R))$ into $Z$, such that
the pushed-forward currents converge in the flat sense in $Z$:
\begin{equation*}
\varphi_{j\#} (T_j \llcorner \bar B(p_j,R)) \to \varphi_{\#} (T \llcorner \bar B(q,R)) \text{ in } d_Z^F.
\end{equation*}
Let $u$ be a Lipschitz function on $(X,d)$ that is negative in the interior of $E$, positive outside $E$, and zero on $\partial E$. 
Let $U:Z \to \R$ be a Lipschitz extension of $u \circ \varphi^{-1}$, as in \eqref{eqn_U}, with $\Lip(U)=\Lip(u)$, and define
 $u_j = U \circ \varphi_j$, which define Lipschitz functions on $\bar B(p_j,R)$ with $\Lip(u_j) \leq \Lip(u)$. For $\delta \in \R$ and $r \in (0,R]$, define the 
sets
\begin{align}
E_j^{\delta} &= \{u_j \leq \delta\} \subseteq \bar B(p_j,R) \subset M_j, \nonumber\\
E_j^{\delta,r} &= \{u_j \leq \delta\} \cap \bar B(p_j,r) \subset M_j. \label{E_j_delta_r}
\end{align}
That is, $E_j^{\delta,r}$ is a truncation of $E_j^{\delta}$, which will be needed in case $E_j^{\delta}$ extends all the way out to $\partial B(p_j,R)$, as shown in Figure \ref{fig_E_j}. Also define $E^\delta = \{u \leq \delta\}$ as before, so that $E^0=E$.

\begin{figure}[ht]
	\begin{center}
		\includegraphics[scale=0.7]{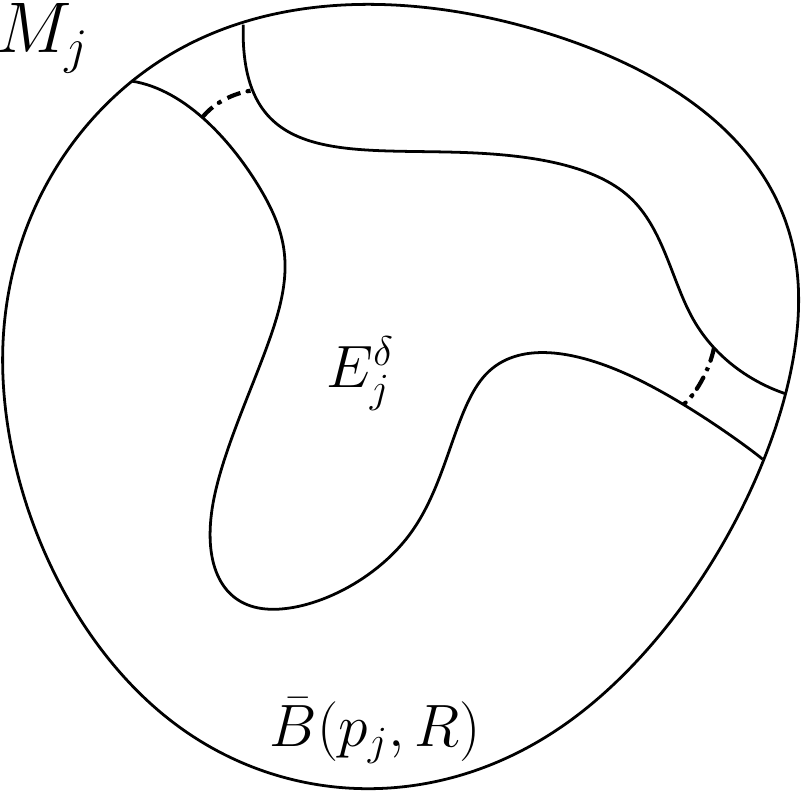}
		\caption{\small The ball $\bar B(p_j, R) \subset M_j$ is shown, along with the subset $E^\delta_j$. The portion of $E^\delta_j$ bounded by the dotted lines represents $E^{\delta,r}_j$, for $r < R$.
			\label{fig_E_j}}
	\end{center}
\end{figure}

\smallskip

We prove two results regarding the convergence of the boundary masses of $T_j \llcorner E_j^{\delta,r}$ to the boundary mass of $T \rst E^\delta$. In the first case for the defining function $u$ for $E$, $|du|$ is assumed to be exactly 1 near $\partial E$; in the second, $|du|$ need only be close to 1 near $\partial E$.
\begin{prop}
\label{prop_perimeters_converge}
Suppose $\Lip(u)=1$ and that $u$ is $C^1$ on a neighborhood of $\partial E$, and $|du|_g = 1$ on this neighborhood. Then there exists $\delta_0 < 0$ such that, upon  passing to a subsequence,
\begin{align}
N_j \llcorner E_j^{\delta,r} &\toF N \llcorner E^\delta, \label{eqn_flat_limit}\\
\lim_{j \to \infty} \M (T_j \llcorner E_j^{\delta,r}) &= \M(T \llcorner E^\delta), \label{eqn_mass_limit}  \qquad \text{and}\\
\liminf_{j \to \infty} \M(\partial (T_j \llcorner E_j^{\delta,r})) &= \M(\partial (T \llcorner E^\delta))\label{eqn_perim_limit}
\end{align}
for almost all $r \in (R-1,R)$ and almost all $\delta \in [\delta_0, 0]$.
\end{prop}

We  also prove an approximate version of the proposition, allowing $|du|_g$ to be close to 1 near $\partial E$ at the expense of only obtaining approximate convergence of the boundary masses. 
\begin{prop}
\label{prop_perimeters_converge2}
Take the value $\gamma_0>0$ from Lemma \ref{lemma_mass_measure}. For any $\gamma \in [0, \gamma_0]$, if $\Lip(u) \leq (1+\gamma)$, and $u$ is $C^1$ in a neighborhood of $\partial E$, and $(1+\gamma)^{-1} \leq |du|_g \leq (1+\gamma)$ in this neighborhood, then there exists $\delta_0 < 0$ such that, upon  passing to a subsequence,
\begin{align}
N_j \llcorner E_j^{\delta,r} &\toF N \llcorner E^\delta, \label{eqn_flat_limit2}\\
\lim_{j \to \infty} \M (T_j \llcorner E_j^{\delta,r}) &= \M(T \llcorner E^\delta), \label{eqn_mass_limit2} \qquad \text{and}\\
 \M(\partial (T \llcorner E^\delta))&\leq \liminf_{j \to \infty} \M(\partial (T_j \llcorner E_j^{\delta,r}))  \leq (1+\gamma)^2 \M(\partial (T \llcorner E^\delta)) \label{eqn_perim_limit2}
\end{align}
for almost all $r \in (R-1,R)$ and a dense subset of $\delta \in [\delta_0, 0]$.
\end{prop}
In fact we will only use Proposition \ref{prop_perimeters_converge2} later in the paper, but we present Proposition \ref{prop_perimeters_converge} and its proof separately for the sake of exposition (and possibly other applications).

\begin{remark}
To simplify the notation, $T$ in this proof and the next will refer to the integral current $T \llcorner \bar B(q,R)$ on $(\bar X, \bar d)$, and $T_j$ will refer to the integral current $T_j \llcorner \bar B(p_j,R)$ on $M_j$. $T_j^r$ will denote $T_j \llcorner \bar B(p_j,r)$, for $r>0$.
\end{remark}

\begin{proof}[Proof of Proposition \ref{prop_perimeters_converge}]
We first claim there exists $\delta_0<0$ so that (upon passing to a subsequence) for almost all $\delta_1 < \delta_2$ in $[\delta_0, 0]$ and all $r \in [R-1,R]$, we have
\begin{equation}
\label{eqn_limsup}
\limsup_{j \to \infty} \int_{\delta_1}^{\delta_2} \M \langle T_j^r, u_j,s\rangle ds \leq \int_{\delta_1}^{\delta_2} \M \langle T,u,s\rangle ds.
\end{equation}
At first, fix  any $r \in [R-1,R]$ and any pair $\delta_1 < \delta_2 \leq 0$.

To establish (\ref{eqn_limsup}), we  first apply  Theorem \ref{thm_slicing} to obtain a mass estimate. Recall $A_j^{\delta_1, \delta_2} = \{\delta_1 \leq u_j \leq \delta_2\}$ and $A^{\delta_1, \delta_2}=\{\delta_1 \leq u \leq \delta_2\}$.
\begin{align}
\int_{\delta_1}^{\delta_2} \M\langle T_j^r ,u_j,s\rangle ds &= \int_{-\infty}^{\infty} \M\langle T_j^r \rst A_j^{\delta_1, \delta_2} ,u_j,s\rangle ds\nonumber\\
&\leq \Lip(u_j) \mathbb{M}(T_j^r \rst A_j^{\delta_1, \delta_2})\nonumber\\
&\leq  \mathbb{M}(T_j \rst A_j^{\delta_1, \delta_2} ), \label{eqn_annulus_j}
\end{align}
since $\Lip(u_j) \leq 1$.

We can use a similar argument along with Lemma \ref{lemma_mass_measure} to compute (not just estimate) the mass of $T \llcorner A^{\delta_1, \delta_2}$. Fix $\delta_0<0$ small enough in absolute value so that $\{\delta_0 \leq u \leq 0\}$ lies in $X \setminus K$ and the region on which $u$ is $C^1$ with $|du|_g=1$.
Then for any $\delta_1 < \delta_2$ in $[\delta_0, 0]$, using Theorem \ref{thm_slicing},
\begin{align}
\int_{\delta_1}^{\delta_2} \M \langle T, u, s\rangle ds &=\int_{-\infty}^\infty\M \langle T \llcorner A^{\delta_1,\delta_2}, u, s\rangle ds\nonumber \\
&=\mathbb{M}((T \rst A^{\delta_1, \delta_2}) \rst du)\nonumber \\
&=\mathbb{M}(T \rst A^{\delta_1, \delta_2}). \label{eqn_annulus}
\end{align}
Here, we used the fact that $\|T \rst du\| = \|T\|$ as Borel measures on $A^{\delta_1,\delta_2}$ by Lemma \ref{lemma_mass_measure}.

Combining \eqref{eqn_annulus_j} and \eqref{eqn_annulus} using Lemma \ref{lemma_volumes}, part (b) (which may involve passing to a subsequence and restricting to almost all $\delta_1< \delta_2$) proves the claim, (\ref{eqn_limsup}). 

Next, we claim that, upon passing to a further subsequence,
\begin{equation}
\label{eqn_limsup2}
\limsup_{j \to \infty} \int_{\delta_1}^{\delta_2} \M (\partial (T_j^r \llcorner E_j^s))ds \leq \int_{\delta_1}^{\delta_2} \M (\partial (T \llcorner E^s))ds
\end{equation}
for \emph{almost} all $r \in (R-1,R)$ and almost all $\delta_1 < \delta_2$ in $[\delta_0, 0]$. By the definition of slice,
$$\langle T_j^r ,u_j,s\rangle = \partial (T_j^r \llcorner \{u_j \leq s\}) - (\partial T_j^r) \llcorner \{u_j \leq s\}.$$
Similarly,
\begin{align*}
\langle T ,u,s\rangle &= \partial (T \llcorner \{u \leq s\}) - (\partial T) \llcorner \{u \leq s\}\\
&= \partial (T \llcorner \{u \leq s\}),
\end{align*}
provided $s \leq \alpha$, for some $\alpha>0$, chosen so that $E^\alpha \subseteq B(q,R-1)$. (This is because the locally integral current underlying $M$ has zero boundary, and $\{u \leq s\} \subseteq B(q,R-1)$ is away from $\partial B(q,R)$ for $s>0$ sufficiently small.)
Using (\ref{eqn_limsup}) then we can establish (\ref{eqn_limsup2}) by showing:
\begin{equation}
\label{eqn_masses_zero}
\lim_{j \to \infty} \M\left((\partial T_j^r) \llcorner \{u_j \leq 0\} \right) = 0
\end{equation}
for almost all $r \in (R-1,R)$, upon passing to a subsequence. (The intuition behind \eqref{eqn_masses_zero} is that since $E^\delta$ has zero volume outside of radius $R-1$ and we have $\VF$ convergence, the volume of $E^\delta_j$ between radii $R-1$ and $R$ should converge to zero (as will be proved in \eqref{eqn_M_T_j}); thus, by slicing, \eqref{eqn_masses_zero} must hold. See Figure \ref{fig_E_j_annulus}.)

\begin{proof}[Proof of (\ref{eqn_masses_zero})]
 Recall $p_j \to q$ is assumed as part of the setup. 
 First we claim that 
 $$\varphi_{j\#}(T_j \llcorner (E^\delta_j \cap A(p_j, r_1,r_2))) \to \varphi_{\#}(T \llcorner (E^\delta \cap A(q,r_1,r_2))) \text{ in } d_Z^F$$
 for almost all $r_1 < r_2$ in $(R-1,R)$ and almost all $\delta \in [\delta_0,\alpha]$, upon passing to a subsequence. See Figure \ref{fig_E_j_annulus}. To see this, recall from Lemma \ref{lemma_convergence_balls_annuli}(b) that
for almost all $r_1<r_2$
$$\varphi_{j\#}(T_j \llcorner A(p_j,r_1,r_2)) \to  \varphi_{\#}(T \llcorner A(q,r_1,r_2))$$
in $d_Z^F$ (taking a subsequence). Then Lemma \ref{lemma_K_A_converge} part (a), applied with $T_j \llcorner A(p_j,r_1,r_2)$ in place of $T_j$, establishes the claim. 

\begin{figure}[ht]
	\begin{center}
		\includegraphics[scale=0.7]{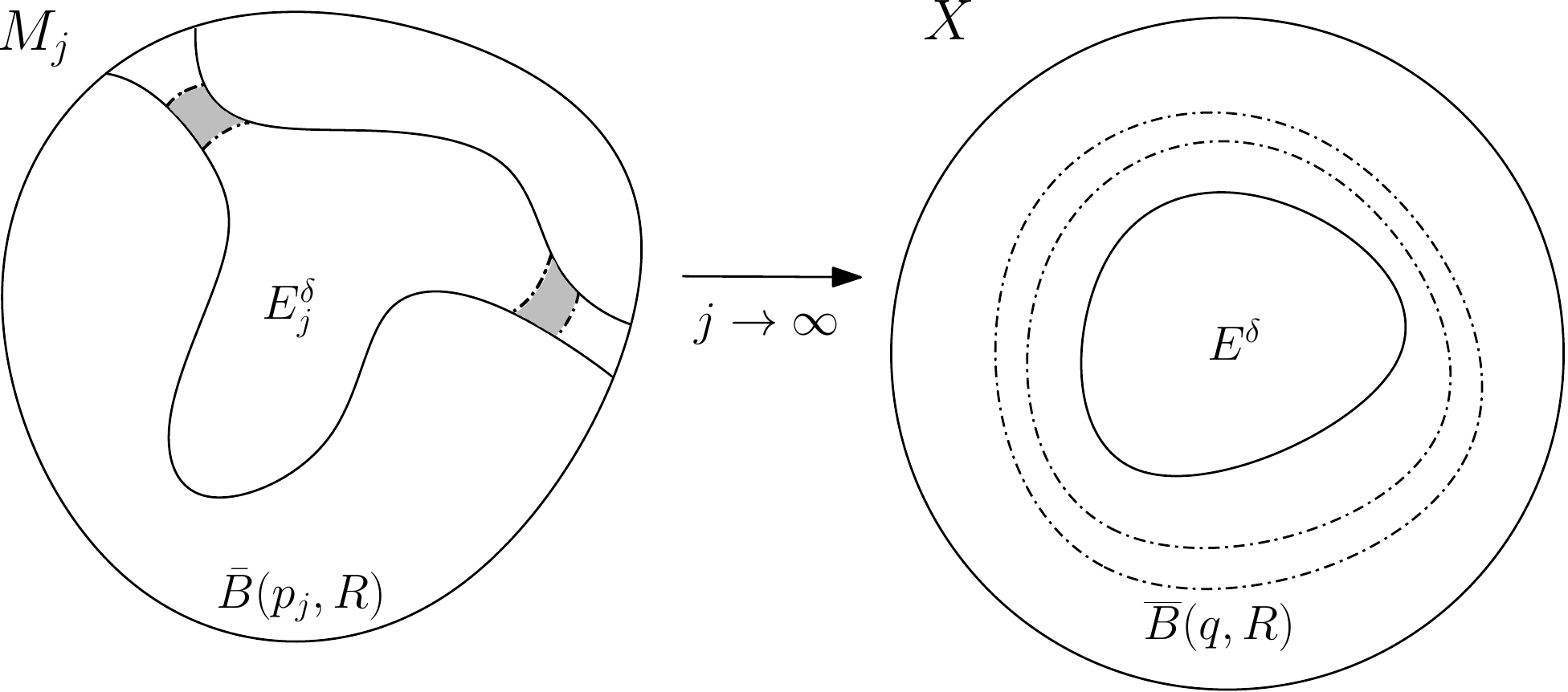}
		\caption{\small On the left the ball $\bar B(p_j, R) \subset M_j$ is shown, along with the subset $E^\delta_j$. The shaded region represents $E^\delta_j \cap A(p_j, r_1, r_2)$. On the right $\bar B(q, R) \subset X$ is shown. The set $E^\delta \cap A(q, r_1, r_2)$ is empty, where the region bounded by the dotted lines represent the annlus $A(q, r_1, r_2)$; this gives intuition for \eqref{eqn_M_T_j}.
			\label{fig_E_j_annulus}}
	\end{center}
\end{figure}

 Next, note that for  $r_1< r_2$ in $(R-1,R)$ and $\delta \leq \alpha$,  $T \llcorner (E^\delta \cap A(q, r_1,r_2)) = 0$ by the choice of $\alpha$ above (again, see Figure \ref{fig_E_j_annulus}). By $\VF$ convergence and Lemma \ref{lemma_volumes_general},  we then have
\begin{equation}
\label{eqn_M_T_j}
\M(T_j \llcorner (E^\delta_j \cap A(p_j, r_1,r_2))) \to 0
\end{equation}
for almost all values $r_1 < r_2$ in $(R-1,R)$ and almost all $\delta \in [\delta_0,\alpha]$.
By Theorem \ref{thm_slicing}, this implies
$$\int_{r_1}^{r_2} \M\langle T_j \llcorner E_j^\delta, \rho_j,r\rangle dr \to 0,$$
for almost all $\delta \in [\delta_0,\alpha]$, where $\rho_j$ is the distance from $p_j$ with respect to $d_{g_j}$. Thus, we may pass to an ($s$-dependent) subsequence so that
$$\M\langle T_j \llcorner E_j^s, \rho_j,r\rangle \to 0 \text{  pointwise a.e. in } r \in [r_1,r_2] \text{ as } j \to \infty;$$
pass to such a subsequence for some fixed $s \in [0,\alpha]$.
From the definition of slice and Lemma \ref{lemma_leibniz}(b),
\begin{align*}
\langle T_j \llcorner E_j^s, \rho_j,r\rangle &= \partial((T_j \llcorner E_j^s) \llcorner \{\rho_j \leq r\}) - (\partial (T_j \llcorner E_j^s)) \llcorner \{\rho_j \leq r\}\\
   &= \partial((T_j \llcorner \{u_j \leq s\}) \llcorner \{\rho_j \leq r\}) - (\partial (T_j \llcorner \{u_j \leq s\})) \llcorner \{\rho_j \leq r\}\\
   &= \partial(T_j \llcorner \{\rho_j \leq r\}) \llcorner \{u_j \leq s\} - (\partial T_j) \llcorner \{\rho_j \leq r\} \llcorner \{u_j \leq s\}\\
   &= (\partial T_j^r) \llcorner \{u_j \leq s\}, 
\end{align*}
if $r <R$, having used the fact that $\partial T_j \llcorner  \{\rho_j \leq r\}  = 0$ (since $\partial T_j$ is supported in $\partial B(p_j,R)$). In particular,
$$\lim_{j \to \infty} \M((\partial T_j^r) \llcorner \{u_j \leq s\}) = 0$$
for almost all $r \in [r_1,r_2]$. Since $s \geq 0$, this implies
$$\lim_{j \to \infty} \M((\partial T_j^r )\llcorner \{u_j \leq 0\}) = 0$$
for almost all $r \in [r_1,r_2]$.
Since the above argument holds for almost all $r_1 < r_2$ in $(R-1,R)$, we have established \eqref{eqn_masses_zero} for almost all $r \in (R-1,R)$.
\end{proof}

With \eqref{eqn_masses_zero} proven, \eqref{eqn_limsup2} follows from \eqref{eqn_limsup}.

Finally, we can complete the proof of the proposition, using (\ref{eqn_limsup2}), as follows. We have from Lemma \ref{lemma_convergence_balls_annuli} that, upon passing to a subsequence,
$$\varphi_{j\#}(T_j^r) \to \varphi_{\#} (T \llcorner B(q,r))$$
in $d_Z^F$ for almost all $r \in (0,R)$. Then by Lemma \ref{lemma_K_A_converge}, upon passing to a subsequence,
$$\varphi_{j\#}(T_j^r \llcorner E_j^s) \to \varphi_{\#} ((T \llcorner B(q,r))\llcorner E^s)$$
in $d_Z^F$ for almost all $r \in (0,R)$ and almost all $s \in \R$. But for $r > R-1$ and $s \leq \alpha$, 
$(T \llcorner B(q,r))\llcorner E^s = T \llcorner E^s$ (since $E^s \subseteq B(q,R-1)$). In particular, using Lemma \ref{lemma_volumes_general}, we have
\begin{equation}
\label{eqn_T_j_VF}
\varphi_{j\#}(T_j^r \llcorner E_j^s) \to \varphi_{\#} (T \rst E^s),
\end{equation}
in $d_Z^F$, and their masses converge, for almost all $r \in (R-1,R)$ and almost all $s \leq \alpha$. Since $ T_j^r \llcorner E_j^\delta = T_j \llcorner E^{\delta,r}_j$, this implies \eqref{eqn_flat_limit}, \eqref{eqn_mass_limit} and therefore  that the boundaries converge in $\mathcal{F}$. By lower semicontinuity of $\M$ (i.e., \eqref{eqn_M_LSC})
\begin{equation}
\label{eqn_lsc_bdry_mass}
\liminf_{j \to \infty} \M(\partial (T_j^r \llcorner E_j^s)) \geq \M(\partial (T \llcorner E^s))
\end{equation}
for  almost all $r \in (R-1,R)$ and almost all $s \leq \alpha$.

Fix arbitrary values $\delta_1 < \delta_2$ in $[\delta_0, 0]$ and $r \in (R-1,R)$ for which (\ref{eqn_limsup2}) holds. Suppose that the subset of $[\delta_1, \delta_2]$ on which strict inequality in (\ref{eqn_lsc_bdry_mass}) holds has positive measure. Then by Fatou's Lemma,
\begin{align*}
\liminf_{j \to \infty} \int_{\delta_1}^{\delta_2} \M(\partial (T_j^r \llcorner E_j^s)) ds &\geq  \int_{\delta_1}^{\delta_2}\liminf_{j \to \infty} \M(\partial (T_j^r \llcorner E_j^s)) ds\\
>  \int_{\delta_1}^{\delta_2}\M(\partial (T \llcorner E^s)) ds.
\end{align*}
This contradicts (\ref{eqn_limsup2}).
Since the above arguments hold for almost all $\delta_1 < \delta_2$ in $[\delta_0, 0]$,  \eqref{eqn_perim_limit}  follows, completing the proof Proposition \ref{prop_perimeters_converge}.
\end{proof}

\begin{proof}[Proof of Proposition \ref{prop_perimeters_converge2}]
We largely reuse the proof of Proposition \ref{prop_perimeters_converge} here. In this proof, $\delta_0 < 0$ is chosen analogously, i.e. so that $\{\delta_0 \leq u \leq 0\} \subset X \setminus K$ is contained in the set on which $u$ is $C^1$ and
$(1+\gamma)^{-1} \leq |du|_g \leq (1+\gamma)$ is satisfied.

Using $\Lip(u) \leq (1+\gamma)$, the same reasoning as \eqref{eqn_annulus_j} gives
$$(1+\gamma)^{-1}  \int_{\delta_1}^{\delta_2} \M\langle T_j^r ,u_j,s\rangle \leq \mathbb{M}(T_j \llcorner A_j^{\delta_1,\delta_2}).$$
Using the bounds on $|du|_g$ and Lemma \ref{lemma_mass_measure}, the same reasoning as \eqref{eqn_annulus} gives, for $\delta_1 < \delta_2$ in $[\delta_0,0]$,
$$\mathbb{M}(T \llcorner A^{\delta_1,\delta_2}) \leq (1+\gamma) \int_{\delta_1}^{\delta_2} \M \langle T, u, s\rangle ds.$$
Combining these with Lemma \ref{lemma_volumes}(b), we may pass to a subsequence so that for almost all $\delta_1 < \delta_2$ in $[\delta_0,0]$, we have
\begin{equation}
\label{eqn_limsup3}
\limsup_{j \to \infty} \int_{\delta_1}^{\delta_2} \M \langle T_j^r, u_j,s\rangle ds \leq (1+\gamma)^2 \int_{\delta_1}^{\delta_2} \M \langle T,u,s\rangle ds.
\end{equation}
By the same logic as in the proof of \eqref{eqn_limsup2}, we have
\begin{equation}
\label{eqn_limsup4}
\limsup_{j \to \infty} \int_{\delta_1}^{\delta_2} \M (\partial (T_j^r \llcorner E_j^s))ds \leq (1+\gamma)^2 \int_{\delta_1}^{\delta_2} \M (\partial (T \llcorner E^s))ds
\end{equation}
for almost all $r \in (R-1,R)$ and almost all $\delta_1 < \delta_2$ in $[\delta_0, 0]$, upon passing to a subsequence.  Fix such values of $r,\delta_1$, and $\delta_2$.
Note that  \eqref{eqn_T_j_VF} and \eqref{eqn_lsc_bdry_mass}
continue to hold for  almost all $r \in (R-1,R)$ and almost all $s \leq \alpha$.

Suppose that the set of values $\delta \in [\delta_1, \delta_2]$ for which
$$\liminf_{j \to \infty} \M(\partial (T_j \llcorner E_j^{\delta,r}))  > (1+\gamma)^2 \M(\partial (T \llcorner E^\delta))$$
has full measure. By Fatou's Lemma,

\begin{align*}
\liminf_{j \to \infty} \int_{\delta_1}^{\delta_2} \M(\partial (T_j \llcorner E_j^{s,r}))ds &\geq\int_{\delta_1}^{\delta_2} \liminf_{j \to \infty}  \M(\partial (T_j \llcorner E_j^{s,r}))ds\\
&>(1+\gamma)^2 \int_{\delta_1}^{\delta_2}  \M(\partial (T \llcorner E^s))ds.
\end{align*}
This contradicts \eqref{eqn_limsup4}, so that \eqref{eqn_perim_limit2} holds for a set of positive measure in $[\delta_1, \delta_2]$. Since the above argument applies for almost all $\delta_1 < \delta_2$ in $[\delta_0, 0]$, \eqref{eqn_perim_limit2} follows for a dense set of $\delta$ in $[\delta_0,0]$. Statements \eqref{eqn_flat_limit2} and \eqref{eqn_mass_limit2} follow without modification from the proof of the previous proposition.
\end{proof}

We conclude section \ref{sec_regions} by pointing out that there actually exist functions $u$ satisfying the hypotheses of Proposition \ref{prop_perimeters_converge2}. This is done in Lemma \ref{lemma_u} below. (It is not so clear there exists a function $u$ satisfying the hypothesis of Proposition \ref{prop_perimeters_converge} if $g$ is only $C^0$.) We continue with the setup given at the start of this subsection, \ref{sec_converge_perim}. 	In particular, $g$ and $d$ are locally compatible on $X \setminus K$, and $\partial E \subset X \setminus K$ is smooth.

We first require a lemma:
\begin{lemma}
For $a > 0$, let
	$$\mathcal{U}_a = \left\{x \in X \; : \; d(x,\partial E) < a \right\}.$$
	\begin{enumerate}
		\item[(a)] For all $a > 0$ sufficiently small, $\overline{\mathcal{U}}_a$ is disjoint from $K$.
		\item[(b)] For all $a > 0$ sufficiently small, it holds that for all $x \in \mathcal{U}_a$, $d(x,\partial E)= d_g(x, \partial E)$.
	\end{enumerate}
\end{lemma}
\begin{proof}
	The first part is elementary, since $K$ is (sequentially) compact and $\partial E$ is closed. 
	
	Now, cover $\partial E$ by finitely many balls $B_d(p_i, r_i)$, where $p_i \in \partial E$ and $B_d(p_i, 4r_i)$ is a neighborhood of compatibility for $g$ and $d$ (see Definition \ref{def_locally_compatible}). Let $a>0$ be sufficiently small so that $\mathcal{U}_a$ is contained in $\cup_i B_d(p_i, 2r_i)$ (e.g., $a=\min_i \{r_i\}$).
	
	Clearly $d(x, \partial E) \leq d_g(x, \partial E)$, as in \eqref{eqn_d_d_g}, for all $x \in X \setminus K$. To show the reverse inequality in $\mathcal{U}_a$, let $x \in \mathcal{U}_a$, so that $x$ belongs to some $B_d(p_i, 2r_i)$. Since $\partial E$ is compact, there exists some $q \in \partial E$ such that $d(x,\partial E) = d(x,q)$. Since $p_i$ also belongs to $\partial E$, we have $d(x,p_i) \geq d(x,q)$. Then
	$$d(p_i, q) \leq d(p_i, x) + d(x,q) \leq 2 d(p_i, x) < 4r_i.$$
	Thus, $q \in B_d(p_i, 4r_i)$, a neighborhood of compatibility, so that $d_g(x,q) = d(x,q) = d(x, \partial E)$. By the definition of $d_g(x,\partial E)$, we see that
	$d_g(x, \partial E) \leq d(x, \partial E)$, which completes the proof. 
\end{proof}

\begin{lemma}
\label{lemma_u}
Given any $\gamma>0$, there exists a function $u$ on $X$ with $Lip_d(u) \leq 1+\gamma$, such that $u>0$ in $X \setminus E$, $u \leq 0$ in $E$,  $u^{-1}(0)=\partial E$, and there exists a neighborhood $W$ of $\partial E$ on which $u$ is smooth and
$$(1+\gamma)^{-1} \leq |du|_g \leq 1+ \gamma$$
is satisfied on $W$.
\end{lemma}
Note that $u$ and $W$ depend on $\gamma$.

\begin{proof}
	Initially, we let $u$ be the signed distance function in $X$ to $\partial E$, with respect to $d$, i.e.
	\begin{equation}
	\label{eqn_signed_distance}
	u(x) = \begin{cases}
	d(x, \partial E), & \text{if } x \in X \setminus E\\
	-d(x, \partial E),& \text{if }  x \in E.
	\end{cases}
	\end{equation}
	For $a>0$, note
	$$\mathcal{U}_a = u^{-1}(-a, a).$$
	By the previous lemma, we may fix $a>0$ sufficiently small so that $\overline{\mathcal{U}}_a \subset X \setminus K$ and so that the signed distance function $u_g$ to $\partial E$ on $X \setminus K$  with respect to $d_g$ agrees with $u$ on $\mathcal{U}_a$. Let $\mathcal{U} = \mathcal{U}_a$. Note that $\overline{\mathcal{U}} = u^{-1}[-a, a]$ is compact. If $g$ happens to be smooth, then $u$ is the desired function, and the proof is complete (with $\gamma=0$) (since $u_g$ is smooth with $|d u_g|_g=1$ near $\partial E$).

	For each $p \in \overline{\mathcal{U}}$ there exists $r=r(p)>0$ such that $B_d(p,r) \subset X \setminus K$ is a neighborhood of compatibility for $g$ and $d$. This produces  an open cover and hence a finite subcover of $\overline{\mathcal{U}}$ by such neighborhoods; denote the latter by $\{B_d(p_i, r_i)\}_{i=1}^N$ . Let $\delta > 0$ denote the Lebesgue number of this cover, so that every subset of $\overline{\mathcal{U}}$ of diameter less than $\delta$ is contained in some $B_d(p_i, r_i)$. Let $b = \min\left\{\frac{\delta}{2}, 0.9 a\right\}$, and let  $\mathcal{U}' = u^{-1}(-b,b)$, so that $\overline{\mathcal{U}'} \subset \mathcal{U}$.
	
	Let $\gamma>0$. Let $\{g_i\}$ be a sequence of continuous Riemannian metrics on $X \setminus K$, equalling $g$ on $(X \setminus K) \setminus \mathcal{U}$, with $g_i$ smooth on $\mathcal{U}'$, and such that $g_i$ converges to $g$ in $C^0$ as $i \to \infty$. We fix some $i$ sufficiently large so that 
	\begin{equation}
	\label{eqn_norm_g_i}
	(1+\gamma)^{-1} |v|_g \leq |v|_{g_i} \leq (1+\gamma)|v|_g 
	\end{equation}
	on all tangent vectors $v$ to $X \setminus K$. In particular,
	the distance functions  $d_{g_i}$ and $d_g$ are uniformly equivalent:
	\begin{equation}
	\label{eqn_uniformly_equiv}
	(1+\gamma)^{-1} d_{g}(x,y) \leq d_{g_i}(x,y) \leq (1+\gamma) d_g(x,y)
	\end{equation}
	for all $x,y \in X \setminus K$.
	Let $u_i$ be the signed distance function to $\partial E$ with respect to $d_{g_i}$, defined on $X \setminus K$, as in \eqref{eqn_signed_distance}, with $d_{g_i}$ in place of $d$ . Of course $\Lip_{d_{g_i}}(u_i) =1$, so that by \eqref{eqn_uniformly_equiv}, $\Lip_{d_g} (u_i) \leq 1+\gamma$ on $X \setminus K$. Increasing $i$ if necessary, we can also arrange that
	\begin{equation}
	\label{eqn_u_u_i}
	(1+\gamma)^{-1} |u(x)| \leq |u_i(x)| \leq (1+\gamma)|u(x)|
	\end{equation}
	for all $x \in \mathcal{U}$, by \eqref{eqn_uniformly_equiv} and because $u=u_g$ on $\mathcal{U}$.
	
	We claim that $u_i$ has Lipschitz constant on $\mathcal{U}'$ with respect to $d$ of at most $1+\gamma$; this is not obvious since $d$ and $d_g$ are in general unequal. Let $x,y \in \mathcal{U}'$ with $x \neq y$. We consider two cases. First, if $d(x,y) < \delta$, then the diameter of $\{x ,y\} \subset \mathcal{U}$ with respect to $d$ is less than $\delta$, so that $x,y$ belong to a neighborhood of local compatibility. Then
	$$\frac{|u_i(x) - u_i(y)|}{d(x,y)} = \frac{|u_i(x) - u_i(y)|}{d_g(x,y)} \leq \Lip_{d_g} (u_i) \leq 1+ \gamma.$$
	Second, if $d(x,y) \geq \delta$, then by \eqref{eqn_u_u_i} and the definitions of $b$ and $\mathcal{U}'$,
	$$\frac{|u_i(x) - u_i(y)|}{d(x,y)}  \leq \frac{(1+\gamma)(|u(x)| + |u(y)|)}{\delta}  \leq \frac{(1+\gamma)(2b)}{\delta} < 1+\gamma.$$
	Thus, we have $\Lip_d (u_i|_{\mathcal{U}'}) \leq 1$.
	
	Now, since $g_i$ is smooth near $\partial E$, there exists a neighborhood of $\partial E$, contained in $\mathcal{U}'$ on which the signed distance function $u_i$ is smooth, and $|d u_i|_{g_i} = 1$.
	By \eqref{eqn_norm_g_i}, we have
	$$(1+\gamma)^{-1} \leq  |du_i|_g \leq (1+\gamma)$$
	on this neighborhood.
	
	Note that by definition, $u_i^{-1}(0)= \partial E$. We complete the proof by noting that $u_i|_{\mathcal{U}'}$
	can be extended off of $\mathcal{U}'$ to a function $\tilde u:X \to \R$, with $\Lip_d(\tilde u) \leq 1+\gamma$, such that $u_i^{-1}(0) = \partial E$ and $u_i^{-1} (-\infty, 0] = E$; $\tilde u$ is then the desired function. Explicitly, $\tilde u$ can be constructed by capping off $u_i$ by constants slightly above and below 0.
\end{proof}

\section{Proof of Theorem \ref{thm1}}
\label{sec_main_proof}

Let $N_j$ denote the AF local integral current space corresponding to $(M_j, g_j)$. 
Let $m_j \geq 0$ be the ADM mass of $(M_j, g_j)$, and assume $\mu:=\displaystyle\liminf_{j \to \infty} m_j$ is finite (or else the claim trivially follows). Pass to a subsequence (without changing the indexing notation) so that $\displaystyle\lim_{j \to \infty} m_j = \mu \in [0, \infty)$ and so that $m_j \leq \mu+1$ for all $j$; obviously these statements are preserved upon taking further subsequences.

As in the statement of the theorem, we assume that the isoperimetric constants of $(M_j, g_j)$ are uniformly bounded below by some constant $c>0$. Let $C>0$ be the constant in Theorem \ref{thm17} corresponding to upper bound $\mu+1$ for the ADM mass, upper bound $9\sqrt{\pi}$ for the isoperimetric ratio, and lower bound $c$ for the isoperimetric constants.

We first deal with the case in which $\miso(N) \in (0, \infty)$. (If $\miso(N) \leq 0$, there is nothing to prove; the case in which $\miso(N)=\infty$ will be addressed at the end by a simple modification.) 

Let $\epsilon > 0$. Choose a number $\nu_0>0$ sufficiently large so that
\begin{equation}
\label{eqn_nu_0}
\frac{2^{1/3}C}{(c\nu_0)^{1/3}} < \frac{\epsilon}{3}
\end{equation}
and
\begin{equation}
\label{eqn_nu_0'}
\left(\frac{c\nu_0}{2}\right)^{2/3} \geq 36 \pi (\mu+1)^2.
\end{equation}

Fix a compact set $K \subset X$ as in Definition \ref{def_AF_LICS}. Using the definition of isoperimetric mass, we take a compact set $E \subset X$ containing $K$, with $\partial E$ contained in the open AF end $X \setminus K$, so that
\begin{equation}
\label{eqn_iso_mass_est}
\miso(N) < \miso (E)+\frac{\epsilon}{3},
\end{equation}
and the $g$-volume $|E|$ is at least $\nu_0 + 2$. Note that without loss of generality, by \cite[Lemma 16]{JL}, we may use compact sets in place of open sets in the definition of $\miso$, assume $E$ has isoperimetric ratio $I(E,g)$  at most $7\sqrt{\pi}$ and that $\partial E$ is smooth (see also Lemma \ref{lemma_approx_perimeter} and Corollary \ref{cor_smoothing} in the appendix). 
Here, the perimeter of $E$ is well-defined since $\partial E \subset X \setminus K$, and the $g$-volume of $E$ is interpreted as $\|T\|(E) = \|T\|(K) + |E \setminus K|_g$, by Lemma \ref{lemma_mass_measure}(a).

Let $A_0 = |\partial^* E|$ and $V_0= |E|$. In particular,
$$ \miso (E) = \frac{2}{A_0} \left( V_0 - \frac{1}{6\sqrt{\pi}} A_0^{3/2}\right)$$
and $I(E,g) = A_0^{3/2}/ V_0 \leq 7 \sqrt{\pi}$.
Choose a real number $\gamma>0$ sufficiently small, so that for all real numbers $A$ within a factor of $(1+\gamma)^4$ of $ A_0$ and all real numbers $V$ within a factor of $(1+\gamma)^2$ of $V_0$, we have
\begin{equation}
\label{eqn_miso_estimate}
\miso ( E) - \frac{\epsilon}{3} < \frac{2}{A} \left( V - \frac{1}{6\sqrt{\pi}} A^{3/2}\right)<  \miso ( E) + \frac{\epsilon}{3},
\end{equation}
and 
$$A^{3/2} / V \leq 8\sqrt{\pi}.$$ 
If necessary, shrink $\gamma$ so that it is less than the value $\gamma_0$ in Lemma \ref{lemma_mass_measure}(b) and Proposition \ref{prop_perimeters_converge2}.

By the definition of pointed intrinsic flat volume convergence, we may choose a value $R > 0$ so that $B(q,R-1)$ contains $E$,  $N_j \rst B(p_j,R) \toVF N\rst B(q,R)$, and $p_j \to q$.

 Return to the setup of section \ref{sec_converge_perim}, with the choice of a function $u$ as in Lemma \ref{lemma_u} (for the value of $\gamma$  and set $E$ we have chosen). Define the functions $u_j:M_j \to \R$  as in \eqref{eqn_u_j} and the sets $E_j^{\delta,r} \subset M_j$ and $E^\delta \subset X$ (see \eqref{E_j_delta_r} and \eqref{eqn_E_delta}).

Apply Proposition \ref{prop_perimeters_converge2} to obtain an appropriate value $\delta_0<0$.  Since $u$ is $C^1$  with $|du|_g \neq 0$ near $\partial E$ and $g$ is continuous, we have $|E^\delta|$ and $|\partial^* E^\delta|$ are continuous in $\delta$ near $\delta=0$.  In particular, we can shrink $|\delta_0|$ if necessary to be sure that
\begin{align}
|E^\delta| &\geq \nu_0+1, \nonumber\\
(1+\gamma)^{-1} V_0 &\leq |E^\delta| \leq (1+\gamma) V_0, \label{eqn_E_delta_vol}\\
(1+\gamma)^{-1} A_0 &\leq |\partial^* E^\delta| \leq (1+\gamma) A_0, \label{eqn_E_delta_perim} \text{ and}\\
\partial E^\delta &\subset X \setminus K
\end{align} 
 for all $\delta \in [\delta_0, 0]$. Now, fix a value of $\delta \in [\delta_0,0]$ and $r \in (R-1,R)$ (and pass to a subsequence) so that the conclusions of Proposition \ref{prop_perimeters_converge2} hold, and pass to a further subsequence so that \eqref{eqn_perim_limit2} holds with a limit in lieu of a liminf, i.e.
$$\M(\partial (T \llcorner E^\delta)) \leq \lim_{j \to \infty} \M(\partial (T_j \llcorner E_j^{\delta,r}))  \leq (1+\gamma)^2 \M(\partial (T \llcorner E^\delta)).$$
(Here, as in the proofs of Propositions \ref{prop_perimeters_converge} and \ref{prop_perimeters_converge2}, $T$ and $T_j$ represent the restrictions of the local integral currents on $M_j$ and $M$, restricted to $B(p_j,R)$ and $B(q,R)$.)

Truncating finitely many terms, we can be sure that
$$ (1+\gamma)^{-1} \M(\partial (T \llcorner E^\delta)) \leq\M(\partial (T_j \llcorner E_j^{\delta,r}))\\
  \leq (1+\gamma)^3 \M(\partial (T \llcorner E^\delta)),$$
for all $j$. Then, using the equivalence of boundary mass and perimeter (Lemma \ref{lemma_perimeter_bdry_mass}) as well as \eqref{eqn_E_delta_perim}, this implies
$$ (1+\gamma)^{-2}A_0 \leq |\partial^* E_j^{\delta,r}|_{g_j} \leq (1+\gamma)^4 A_0,$$
for all $j$.

Similarly, since the masses converge, we can truncate finitely many terms to arrange that, as in \eqref{eqn_mass_limit2},
$$ (1+\gamma)^{-1} \M(T \rst E^{\delta}) \leq \M(T_j \rst E_j^{\delta,r}) \leq (1+\gamma) \M(T \rst E^{\delta}) ,$$
holds for all $j$. Using the equivalence of volume and mass (Lemma \ref{lemma_mass_measure}(a)) and \eqref{eqn_E_delta_vol}, we therefore have
$$(1+\gamma)^{-2} V_0  \leq |E_j^{\delta,r}|_{g_j} \leq (1+\gamma)^2 V_0$$
for all $j$. Finally, we can also truncate finitely many terms if necessary to arrange $\M(T_j \rst E_j^{\delta,r} )$ (which equals $ |E_j^{\delta,r}|_{g_j}$) is at least $\nu_0$ for all $j$, since $|E^\delta| \geq \nu_0+1$.  In particular, $E^{\delta,r}_j$ satisfies the desired volume and perimeter bounds relative to $V_0$ and $A_0$ (i.e., has perimeter within a factor of $(1+\gamma)^4$ of $A_0$ and volume within a factor of $(1+\gamma)^2$ of $V_0$), and has volume at least $\nu_0$. Then  by \eqref{eqn_miso_estimate},
\begin{equation}
\label{eqn_miso_T_j}
\miso ( E_j^{\delta,r}, g_j) > \miso ( E,g) - \frac{\epsilon}{3},
\end{equation}
and $I(E_j^{\delta,r}, g_j) \leq 8\sqrt{\pi}$.

At this point, to simplify notation, let $E_j$ denote $E_j^{\delta,r}$ (for the values of $\delta$ and $r$ fixed above). Now, $E_j\subset M_j$ is not necessarily a smooth set, but it is compact, and we can perturb it to a compact subset $\hat E_j$ of $M_j$ with smooth boundary (for each $j$), changing the volume and perimeter, and hence the isoperimetric mass and isoperimetric ratio, by arbitrarily small amounts (Lemma \ref{lemma_approx_perimeter} in the appendix). In particular, we can arrange that
\begin{align}
|E_j|_{g_j} &<|\hat E_j|_{g_j} + \frac{1}{j}, \nonumber\\
\miso (E_j, g_j) &< \miso ( \hat E_j,g_j) + \frac{1}{j}, \text{ and} \label{eqn_miso_hat_E_j} \\
I(\hat E_j,g_j) & \leq 9\sqrt{\pi}. \nonumber
\end{align}

Now, let $\tilde E_j$ be the outermost minimizing hull\footnote{Recall that any bounded open set $\Omega$ in a smooth asymptotically flat 3-manifold $M$ admits a unique outermost minimizing hull, i.e. a bounded open set $\tilde \Omega \subset M$ with the least perimeter among all bounded open sets containing $\Omega$, and containing any other least-perimeter such sets. If $\partial \Omega$ is smooth, then $\partial \tilde \Omega$ is $C^{1,1}$. We refer the reader to \cite[section 1]{HI} for further details.} of $\hat E_j$ in $(M_j, g_j)$. Since $\tilde E_j$ has at least as much volume and as most as much perimeter as $\hat E_j$, we have
\begin{equation}
\label{eqn_miso_tilde_E_j} 
\miso (\hat E_j,g_j) \leq \miso (\tilde E_j,g_j).
\end{equation}
and
$$I(\tilde E_j,g_j) \leq I(\hat E_j,g_j) \leq 9\sqrt{\pi}.$$

The next step is to apply Theorem \ref{thm17} to $(M_j, g_j)$ with region $\tilde E_j$. All of the hypotheses clearly hold with the ADM mass upper bound of $\mu+1$, isoperimetric ratio upper bound of $9\sqrt{\pi}$, and isoperimetric constant lower bound of $c$ given at the start of the proof, though we must check that 
\begin{equation}
\label{eqn_tilde_E_j}
|\partial \tilde E_j|_{g_j} \geq 36\pi (\mu+1)^2.
\end{equation}
 To show this, by the isoperimetric inequality (with the uniform lower bound $c$ on isoperimetric constants)
\begin{equation}
\label{eqn_area_bound}
|\partial \tilde E_j|_{g_j}^{3/2} \geq c |\tilde E_j|_{g_j} \geq c |\hat E_j|_{g_j} > c\left( |E_j|_{g_j} - \frac{1}{j}\right)  \geq c\left( \nu_0 - \frac{1}{j}\right)
\end{equation}
In particular, after truncating finitely many terms,
\begin{equation}
\label{eqn_c_nu}
|\partial \tilde E_j|^{3/2} \geq \frac{c\nu_0}{2}.
\end{equation}
With (\ref{eqn_nu_0'}), this establishes \eqref{eqn_tilde_E_j}.

Thus, by Theorem \ref{thm17} with the constant $C$ determined at the start of the proof, we have
$$\miso (\tilde E_j) \leq m_j + \frac{C}{\sqrt{|\partial \tilde E_j}|_{g_j}}$$
Finally, by (\ref{eqn_nu_0}) and \eqref{eqn_c_nu}, this immediately gives:
\begin{equation}
\label{eqn_m_j}
\miso (\tilde E_j) < m_j + \frac{\epsilon}{3}.
\end{equation}

Combining \eqref{eqn_iso_mass_est}, \eqref{eqn_miso_T_j}, \eqref{eqn_miso_hat_E_j}, \eqref{eqn_miso_tilde_E_j}, and \eqref{eqn_m_j}, we have
$$\miso(N) \leq m_j + \epsilon + \frac{1}{j}.$$
Taking $\lim_{j \to \infty}$ now proves Theorem \ref{thm1} in the case that $\miso(N) < +\infty$, since $\epsilon$ was arbitrary.

In the case that $\miso(N)=+\infty$, a similar argument works, upon replacing \eqref{eqn_iso_mass_est} with
$$\miso (E)> \epsilon^{-1},$$
which leads to contradiction of the assumption that $\displaystyle \liminf_{j \to \infty} m_j$ was finite.
This completes the proof of Theorem \ref{thm1}. \qed

\begin{remark} We discuss here the hypotheses of Theorem \ref{thm1}.

Nonnegativity of scalar curvature and the absence of compact minimal surfaces are both necessary, even for pointed $C^2$ Cheeger--Gromov convergence \cite{Jau}.

The  lower bound on isoperimetric constants is used in two places: in the application of Theorem \ref{thm17} and in \eqref{eqn_area_bound}. It would be interesting to remove this hypothesis; we have no examples to indicate it is necessary. We point out that the isoperimetric inequality is only used for regions whose perimeters are small relative to the mass scale, i.e. for perimeters less than $36\pi (\mu+1)^2$, $\mu$ being the limit of the masses. This is apparent for \eqref{eqn_area_bound}: if this inequality holds already, there is no need to use the isoperimetric inequality. In Theorem \ref{thm17}, $C$ only depends on the isoperimetric constant for regions of area at most $36\pi (\mu+1)^2$ (see the proofs of Lemma 32, Lemma 33, and Theorem 17 in \cite{JL}). 

It would also be interesting to investigate whether pointed $\VF$ convergence can be replaced with pointed $\F$ convergence, as we again have no examples to suggest volume convergence is necessary. Volume convergence is needed for our proof, however, to establish the almost-convergence of perimeters (Proposition \ref{prop_perimeters_converge2}).

Finally, it is likely possible to generalize Theorem \ref{thm1} to a broader definition of asymptotically flat local integral current space, e.g., without assuming a manifold structure at infinity, but we do not pursue this here.
\end{remark}

\section*{Appendix: perimeter and boundary mass}
Here we recall the definition of the perimeter of a set in a Riemannian manifold, including the case in which the metric is only $C^0$. We also prove Lemma \ref{lemma_perimeter_bdry_mass}, giving the equality of perimeter and boundary mass.

We first recall some basic facts regarding the variation of a function. These concepts are typically stated in the setting of Euclidean space (see \cite{Amb} for instance), but generally have analogs to smooth Riemannian manifolds (see \cite{Mir} for instance, which we follow below).

Let $(M,g)$ be a smooth Riemannian manifold (possibly with boundary) of dimension $m$, and let $f \in L^1(M,\mu_g)$ be a Borel function (where $\mu_g$ is the Riemannian volume measure induced by $g$). The \emph{variation of $f$} is the quantity
\begin{equation}
|Df|_g(M) =  \sup_{\phi} \left\{\int_M f\Div_g\! \phi\; d\mu_g \;\Big|\; \phi \in \Gamma^1_c(TM), |\phi|_g \leq 1  \right\} \in [0,\infty], \label{eqn_Df_M}
\end{equation}
where  $\Gamma^1_c(TM)$ denotes the space of $C^1$ vector fields with compact support on $M$. For example, if $f$ happens to be $C^1$, then $|Df|_g(M) = \int_M |\nabla f|_g d\mu_g$.
We say $f$ has \emph{bounded variation} (with respect to $g$) if $|Df|_g(M)$ is finite. In this case, there exists a finite Radon measure on $M$, denoted $|Df|_g$, and  a $|Df|_g$-measurable vector field $\sigma_f$ on $M$, with $|\sigma_f|_g = 1$ a.e. (with respect to $|Df|_g$), so that
\begin{equation}
\label{eqn_deriv}
\int_M  f\Div_g\! \phi\; d\mu_g = -\int_M g(\sigma_f, \phi) d|Df|_g
\end{equation}
for all $\phi \in  \Gamma^1_c(TM)$. Formula \eqref{eqn_deriv} can be viewed as defining the distributional gradient of a function $f$ of bounded variation. Note that for any open set $U\subseteq M$,
\begin{equation}
|Df|_g(U) =  \sup_{\phi} \left\{\int_U f\Div_g\! \phi\; d\mu_g \;\Big|\; \phi \in \Gamma^1_c(TU), |\phi|_g \leq 1  \right\}, \label{eqn_Df_U}
\end{equation}
consistent with the notation in \eqref{eqn_Df_M}.

If $f$ has bounded variation, it admits smooth approximations in the following sense (see \cite[Proposition 1.4]{Mir}; cf. \cite[Theorem 3.9]{Amb} in the Euclidean case): there exists a sequence $f_i$ of smooth functions on $M$ of compact support, converging to $f$ in $L^1(M,\mu_g)$, such that $|Df_i|_g(M) =\int_M |\nabla f_i|_g d\mu_g\to |Df|_g(M)$ as $i \to \infty$.

We are  interested in the following special case: let $E \subseteq M$ be a Borel set of finite $\mu_g$-measure, i.e. $\chi_E \in L^1(M,\mu_g)$. We say $E$ has \emph{finite perimeter} in $M$ with respect to $g$ if $\chi_E$ has bounded variation with respect to $g$. The \emph{perimeter} of $E$ is then defined to be $|D\chi_E|_g(M)$, which we will also denote in this appendix by $P_g(E)$. From the above approximation result, it can be shown that $E$ can be approximated in volume and perimeter by smooth sets (cf. \cite[Theorem 3.42]{Amb}):

\begin{lemma}
\label{lemma_approx_perimeter}
Suppose $(M,g)$ is a smooth Riemannian manifold. 
Given a set $E \subseteq M$ of finite perimeter, there exists a sequence $E_i$ of open sets with smooth boundary in $M$, such that $\mu_g(E_i \triangle E) \to 0$ and $P_g(E_i) \to P_g(E)$ as $i \to \infty$. If $E$ is precompact, the $E_i$ may be chosen to be precompact.
\end{lemma}

If the Riemannian metric $g$ is only $C^0$, however, the above discussion no longer holds, because the divergence in \eqref{eqn_Df_M} need not be well-defined. To work around this, we first
show how the data $\sigma_f$ and $|Df|_g$ in \eqref{eqn_deriv} are related with respect to different smooth metrics $g$ on $M$. 
\begin{lemma}
\label{lemma_g1_g2}
Let $f$ be a Borel function on a smooth manifold $M$, and let $g_1$ and $g_2$ be smooth Riemannian metrics on $M$. Then:
\begin{enumerate}
\item[(a)] If $f$ has compact support, then $f$ has bounded variation with respect to $g_1$ if and only if $f$ has bounded variation with respect to $g_2$. 
\item[(b)] If $f$ has bounded variation with respect to both $g_1$ and $g_2$, then $|Df|_{g_1}$ and $|Df|_{g_2}$ are mutually absolutely continuous as Borel measures, and in this case,
\item[(c)] the 1-forms $g_i(\sigma_f^i, \cdot)$ (for $i=1,2$) defined in \eqref{eqn_deriv} with respect to $g_1$ and $g_2$ are pointwise multiples of each other in $T^*M$ almost everywhere. (By (b), ``almost-everywhere'' can be taken with respect to $|Df|_{g_1}$ or $|Df|_{g_2}$.)
\end{enumerate}
\end{lemma}

\begin{proof}
Let $W>0$ be the smooth function on $M$ defined by
$$d\mu_{g_2} = W d\mu_{g_1}.$$
From the characterization of divergence as the Lie derivative of the volume form, we have
\begin{equation}
\label{eqn_divergence}
\Div_{g_2} (\phi) = W^{-1} \Div_{g_1}(W\phi)
\end{equation}
for any $C^1$ vector field $\phi$ on $M$. Statement (a) follows from this and the definition of variation, using the fact that $g_1$ and $g_2$ have relative $C^1$ bounds on any compact set. 

Now, assume that $f$ has bounded variation with respect to $g_1$ and $g_2$.
For $i=1,2$, let
$$\alpha_i(\cdot) = g_i(\sigma^i_f,\cdot),$$
which are $|Df|_{g_i}$-measurable 1-forms on $M$, of unit length with respect to $g_i$, $|Df|_{g_i}$-almost everywhere. Using \eqref{eqn_deriv} and \eqref{eqn_divergence}, we have
\begin{align}
-\int_M \alpha_2(\phi) d|Df|_{g_2} &= \int_M f \Div_{g_2}\!(\phi) d\mu_{g_2} \nonumber\\
&= \int_M f \Div_{g_1}\!(W\phi) d\mu_{g_1} \label{eqn_W_phi}\\
&=-\int_M \alpha_1(\phi) Wd|Df|_{g_1} \label{eqn_alpha}
\end{align}
for any $C^1$ vector field $\phi$ on $M$ of compact support. We now prove (b) directly; clearly we need only show one direction. Suppose $A \subset M$ is a Borel set with $|Df|_{g_1}(A)=0$. Suppose first that $A$ is compact. Since $|Df|_{g_1}$ is a Radon measure and is hence outer-regular, given any $\epsilon>0$, there exists a precompact open set $U_\epsilon \subset M$ containing $A$ such that
$|Df|_{g_1}(U_\epsilon) < \epsilon.$ Let $C>0$ be a constant chosen so that $W|\cdot|_{g_1} \leq C |\cdot|_{g_2}$ on tangent vectors based in $\overline{U_\epsilon}$. Then using \eqref{eqn_Df_U} and \eqref{eqn_W_phi}, there exists a $C^1$ vector field $\phi$ supported in $U_\epsilon$ such that $|\phi|_{g_2}\leq 1$ and
\begin{align*}
|Df|_{g_2}(U_\epsilon) &< \epsilon + \int_M f\Div_{g_2}\! \phi\; d\mu_{g_2}\\
&= \epsilon + C\int_M f\Div_{g_1}\! \left(\frac{W\phi}{C}\right)\; d\mu_{g_1}\\
&\leq   \epsilon + C |Df|_{g_1}(U_\epsilon)\\
&\leq \epsilon (1+C).
\end{align*}
Since $\epsilon$ was arbitrary and $C$ can be chosen independently of $\epsilon$ as $\epsilon \to 0$, this shows $|Df|_{g_2}(A)=0$. If $A$ is not compact, this argument together with a simple covering argument suffices to show $|Df|_{g_2}(A)=0$. 
This completes the proof of (b).

From (b), by the Radon--Nikodym theorem, $d|Df|_{g_1} = h d|Df|_{g_2}$ as  measures, for a positive Borel function $h$ on $M$. Combining this with \eqref{eqn_alpha}, we have
\begin{equation}
\label{eqn_measures}
\int_M\left(\alpha_2(\phi)-\alpha_1(\phi) Wh\right) d|Df|_{g_2}  =0
\end{equation}
for any $C^1$ vector field $\phi$ of compact support. This implies that $\alpha_1$ and $\alpha_2$ are pointwise multiples of each other a.e. (with respect to $|Df|_{g_1}$ or $|Df|_{g_2}$).
\end{proof}

The previous lemma allows us to compare the measures $|Df|_g$ with respect to different smooth metrics $g$ that are related by a $C^0$ bound:
\begin{lemma}
\label{lemma_uniform_equiv}
Suppose  $g_1$ and $g_2$ are smooth Riemannian metrics on $M$ of dimension $m$, satisfying
\begin{equation}
\label{eqn_uniform_equiv}
\Lambda^{-1} |\cdot|_{g_1} \leq |\cdot|_{g_2} \leq \Lambda |\cdot|_{g_1}
\end{equation}
on tangent vectors, for some constant $\Lambda\geq 1$. Then
$$\Lambda^{-m-1} |Df|_{g_1} \leq |Df|_{g_2} \leq \Lambda^{m+1}|Df|_{g_1}$$
as Borel measures for any function $f$ on $M$ of bounded variation with respect to both $g_1$ and $g_2$.
\end{lemma}

\begin{proof}
Continuing with the notation in the proof of the previous lemma, consider the 1-forms $\alpha_1$, $\alpha_2$ that are multiples of each other pointwise a.e. and have unit length with respect to $g_1$ and $g_2$, respectively. From \eqref{eqn_uniform_equiv}, this implies
$$\Lambda^{-1} \alpha_1(\cdot) \leq \alpha_2(\cdot) \leq \Lambda \alpha_1(\cdot)$$
as 1-forms a.e. From \eqref{eqn_uniform_equiv} and  the definition of $W$, we have
$$\Lambda^{-m} \leq W \leq \Lambda^m.$$
From these bounds and \eqref{eqn_measures}, it follows that
$$\Lambda^{-m-1} \leq h \leq \Lambda^{m+1}$$
a.e., and from this, the claim follows.
\end{proof}

\begin{cor}
\label{cor_perimeter}
Suppose  $g_1$ and $g_2$ are smooth Riemannian metrics on $M$, satisfying \eqref{eqn_uniform_equiv}. Let $E \subset M$ be a precompact Borel set.
 Then
$$\Lambda^{-m-1} P_{g_1}(E) \leq P_{g_2}(E) \leq \Lambda^{m+1} P_{g_1}(E).$$
\end{cor}
\begin{proof}
Since $E$ is precompact, $|D\chi_E|_{g_1}(M)$ and $|D\chi_E|_{g_1}(M)$ are either both finite or both infinite, by Lemma \ref{lemma_g1_g2}(a). In the former case, the result follows from the previous Lemma with $f=\chi_E$, and in the latter it is trivial.
\end{proof}

At last we can define perimeter with respect to a $C^0$ Riemannian metric $g$ on $M$. Suppose $E \subset M$ is a precompact Borel set. We say $E$ has finite perimeter with respect to $g$ if $E$ has finite perimeter with respect to any smooth Riemannian metric on $M$ (and hence all such metrics, by Lemma \ref{lemma_g1_g2}(a)).  In this case, define
$$P_g(E) = \lim_{i \to \infty} P_{g_i}(E),$$
for any sequence of smooth Riemannian metrics $\{g_i\}$ on $M$, such that $g_i \to g$ in $C^0$. Corollary \ref{cor_perimeter} implies that $P_g(E)$ is well-defined, i.e., is independent of the sequence.

\begin{cor}
\label{cor_smoothing}
The smoothing result in Lemma \ref{lemma_approx_perimeter} holds if $g$ is merely $C^0$, provided the set $E$ is precompact. 
\end{cor}
This follows from Corollary \ref{cor_perimeter} as well.

In the main body of the paper, we use the notation $|\partial^*E|_g$ to denote $P_g(E)$, though we do not require the notion of the reduced boundary $\partial^* E$ itself and so do not discuss it here.

\medskip

We conclude with the proof of Lemma \ref{lemma_perimeter_bdry_mass}, showing the equivalence between perimeter and boundary mass. We restate this here for the reader's convenience:
\begin{lemma}
[Restatement of Lemma \ref{lemma_perimeter_bdry_mass}]
Let $(M,g)$ be a connected, oriented $C^0$ Riemannian manifold of dimension $m$, possibly with boundary. Suppose $d$ is a complete metric on $M$, locally compatible with $g$, and let $E \subseteq M$ be a precompact Borel set. Let $T_E$ be the integer rectifiable $m$-current on $(M,d)$ given by integration over $E$. Then $\M(\partial T_E)$ is finite if and only if $E$ has finite perimeter with respect to $g$, and in this case, 
$$|\partial^* E|_g = \M(\partial T_E).$$
\end{lemma}

\begin{proof}[Proof of Lemma \ref{lemma_perimeter_bdry_mass}]
This proof uses \cite[Theorem 3.7]{AK}, which implies the analogous result on Euclidean space.

If $M$ has a boundary, we may embed $M$ into a smooth manifold $\tilde M$ without boundary (by attaching an open collar neighborhood, for example), and extend $d$ and $g$ accordingly to $\tilde M$. The mass measure $\|\partial T_E\|$ and variation measure $|D \chi_E|_g$ are supported in $M$ and are independent of the extension to $\tilde M$. Thus, it suffices to assume $M$ has no boundary.

Let $\epsilon > 0$, and let $p \in M$.  Using a $g$-orthogonal basis of $T_pM$ along with the continuity of $g$, we can find a coordinate system $(x^i)$ about $p$ on a small precompact neighborhood $U \subset M$ of $p$ such that on $U$:
\begin{align}
d&=d_g \nonumber\\
(1+\epsilon)^{-2} \delta_{ij} &\leq g_{ij} \leq (1+\epsilon)^2 \delta_{ij} \label{eqn_g_d}\\
(1+\epsilon)^{-1} d_\circ(\cdot, \cdot)& \leq d(\cdot, \cdot) \leq (1+\epsilon) d_\circ (\cdot, \cdot)\label{eqn_d_0}
\end{align}
on $U$, where the first condition is possible by local compatibility. We may shrink $U$ if necessary so that it is convex with respect to $\delta_{ij}$; then $d_\circ$, the metric on $U$ induced by the Riemannian metric $\delta_{ij}$, can be regarded as the restriction of the Euclidean metric to $U$.

From \cite[Theorem 3.7]{AK}, we have $|D \chi_{E \cap U}|_{\circ}  \leq \|\partial T_{E \cap U}\|_{\circ} $ as Borel measures on $U$, where the $\circ$ subscript means taken with respect to the Euclidean metric on $U$. Using \eqref{eqn_g_d} and Lemma \ref{lemma_uniform_equiv}, we have $|D \chi_{E \cap U}|_g \leq (1+\epsilon)^{m+1} \|\partial T_{E \cap U}\|_{\circ}$; using \eqref{eqn_d_0} and the fact that the mass measure's metric dependence comes solely from Lipschitz constants, we find further than
$$|D \chi_{E \cap U}|_g \leq (1+ \epsilon)^{2m}  \|\partial T_{E \cap U}\|$$
where the latter is taken with respect to $d$. Since $M$ may be covered by such open neighborhoods $U$, and $|D \chi_E|_g$ and $\|\partial T_E\|$ are Borel measures, we find
$$|D \chi_{E}|_g \leq (1+ \epsilon)^{2m}  \|\partial T_{E }\|$$
as Borel measures on $M$. Since $\epsilon$ was arbitrary, we have $|D \chi_{E}|_g \leq  \|\partial T_{E }\|$.

A similar argument, together with the reverse inequality $\|\partial T_{E \cap U}\|_{\circ} \leq |D \chi_{E \cap U}|_{\circ} $ in \cite[Theorem 3.7]{AK}, completes the proof. 
\end{proof}

\end{document}